\documentclass[12pt]{article}
\usepackage{amsmath,mathdots}
\usepackage{amssymb}
\usepackage{latexsym}
\usepackage{amsthm}
\usepackage{amsmath,amssymb,amsfonts,amsthm}
\usepackage{mdframed}
\usepackage{graphicx}
\usepackage{stmaryrd}
\usepackage{gensymb}

\usepackage{stmaryrd}
\usepackage{multirow}

\usepackage{graphicx}
\usepackage{tikz}
\usetikzlibrary{matrix,arrows}
\usetikzlibrary{positioning}
\usetikzlibrary{fit}
\usetikzlibrary{patterns}

\usepackage[colorlinks,
linkcolor=blue,
anchorcolor=blue,
citecolor=blue
]{hyperref}

\newtheorem{thm}{Theorem}[section]

\newtheorem{conj}[thm]{Conjecture}
\newtheorem{remark}[thm]{Remark}
\newtheorem{exa}[thm]{Example}

\usepackage{graphicx}
\usepackage{tikz}
\usetikzlibrary{matrix,arrows}
\usetikzlibrary{positioning}
\usetikzlibrary{fit}
\usetikzlibrary{patterns}

\newcommand{\dbrac}[1]{{\llbracket #1 \rrbracket}} 	
\newcommand{\boks}[2]{({#1, #2})}   

\newcommand{\pattern}[4]{										
	\raisebox{0.6ex}{
		\begin{tikzpicture}[scale=0.35, baseline=(current bounding box.center), #1]
		\foreach \x/\y in {#4}		\fill[gray!20] (\x,\y) rectangle +(1,1);
		\draw (0.01,0.01) grid (#2+0.99,#2+0.99);
		\foreach \x/\y in {#3}		\filldraw (\x,\y) circle (6pt);
		\end{tikzpicture}}
}

\definecolor{red}{rgb}{1,0,0}
{}

\begin{document}

\begin{center}
{\large \bf  Distributions of mesh patterns of short lengths}
\end{center}

\begin{center}
Sergey Kitaev$^{a}$ and  Philip B. Zhang$^{b}$
\\[6pt]

$^{a}$Department of Computer and Information Sciences \\
University of Strathclyde, 26 Richmond Street, Glasgow G1 1XH, UK\\[6pt]

$^{b}$College of Mathematical Science \\
Tianjin Normal University, Tianjin  300387, P. R. China\\[6pt]

Email:  $^{a}${\tt sergey.kitaev@cis.strath.ac.uk},
           $^{b}${\tt zhang@tjnu.edu.cn}
\end{center}

\noindent\textbf{Abstract.}
A systematic study of {\em avoidance} of mesh patterns of length 2 was conducted by Hilmarsson \emph{et al.}, where 25 out of 65 non-equivalent cases were solved. In this paper, we  give 27 {\em distribution} results  for these patterns including 14 distributions for which avoidance was not known. Moreover, for the unsolved cases, we prove an equidistribution result (out of 6 equidistribution results we prove in total), and  conjecture 6 more equidistributions. Finally, we find seemingly unknown distribution of the well known permutation statistic ``strict fixed point'', which plays a key role in many of our enumerative results.

This paper is the first systematic study of distributions of mesh patterns. Our techniques to obtain the results include, but are not limited to, obtaining functional relations for generating functions, and finding recurrence relations and bijections. 
\\




\section{Introduction}\label{intro}
Patterns in permutations and words have attracted much attention in the literature (see~\cite{Kit} and references therein), and this area of research continues to grow rapidly. 
The notion of a {\em mesh pattern}, generalizing several classes of patterns, was introduced by Br\"and\'en and Claesson \cite{BrCl} to provide explicit expansions for certain permutation statistics as, possibly infinite, linear combinations of (classical) permutation patterns. 
A pair $(\tau,R)$, where $\tau$ is a permutation of length $k$ and $R$ is a subset of $\dbrac{0,k} \times \dbrac{0,k}$, where
$\dbrac{0,k}$ denotes the interval of the integers from $0$ to $k$, is a
\emph{mesh pattern} of length $k$.
Let $\boks{i}{j}$ denote the box whose corners have coordinates $(i,j), (i,j+1),
(i+1,j+1)$, and $(i+1,j)$. Let the horizontal lines represent the values,  and the vertical lines denote the positions in the pattern. Mesh patterns can be drawn by shading the boxes in $R$. The following picture represents
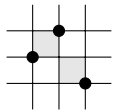
\[
\pattern{scale=1}{3}{1/2,2/3,3/1}{1/2, 2/1}
\]
 the mesh pattern with $\tau=231$ and $R = \{\boks{1}{2},\boks{2}{1}\}$.
Many papers were dedicated to the study of mesh patterns and their generalizations; e.g.\ see \cite{AKV,Borie,JKR,KL,KR1,KRT,T1,T2}. However, the first systematic study of mesh patterns was not done until \cite{Hilmarsson2015Wilf}, where 25 out of 65 non-equivalent {\em avoidance} cases of patterns of length 2 were solved. 
That is, in the 25 cases, the number of permutations avoiding the respective mesh patterns was found.

%

In this paper, we initiate a systematic study of distributions of mesh patterns by giving 27 distribution results  for the patterns considered in \cite{Hilmarsson2015Wilf}, including 14 distributions for which avoidance was not known. Moreover, for the unsolved cases, we prove an equidistribution result (out of 6 equidistribution results we prove in total), and  conjecture 6 more equidistributions (see Table~\ref{tab-2}). 
Techniques we use include generating functions, recurrence relations, and bijections.
 
We note that from the distribution point of view, we  cannot consider just the 65 patterns presented in \cite{Hilmarsson2015Wilf}, since there are more patterns to consider. For example, the pattern Nr.\ 39 =  $\pattern{scale=0.6}{2}{1/1,2/2}{0/1,1/0}$ 
 was considered there, while its Wilf-equivalent pattern (by the {\em Shading Lemma} in \cite{Hilmarsson2015Wilf}) 
 $\pattern{scale=0.6}{2}{1/1,2/2}{0/1,1/0,1/1}$ 
 was not considered. However, these two patterns have different distributions. Two patterns, $p_1$ and $p_2$, are said to be {\em Wilf-equivalent} if for any $n\geq 0$, the number of permutations of length $n$ avoiding $p_1$ is equal to that avoiding $p_2$. 
  
 Table~\ref{tab-1} overviews our enumerative results (27 patterns). In particular,  Nr.\ 3 = $\pattern{scale = 0.6}{2}{1/1,2/2}{0/0,0/1,1/2}$ is a conjectured distribution, Nr.\ 1 = $\pattern{scale = 0.6}{2}{1/1,2/2}{}$ is the well known distribution of non-inversions (same as the distribution of inversions) \cite[p. 21]{Stanley} given over permutations of length $n$ by
 \begin{equation}
 \label{inv-distr}
 (1+q)(1+q+q^2)\cdots(1+q+q^2+\cdots +q^{n-1}),
 \end{equation}
  and Nr.\ 14 = $\pattern{scale = 0.6}{2}{1/1,2/2}{0/1,1/1,1/2,1/0,1/2,2/1}$ is the known distribution of small descents \cite[p. 179]{Charalambides}. 
 
  \begin{table}[!ht]
 	{
 		\renewcommand{\arraystretch}{1.3}
 \begin{center} 
 		\begin{tabular}{|c|c|c||c|c|c|}
 			\hline
 			{Nr.\ } & {Repr.\ $p$}  & {Distribution}  &  {Nr.\ } & {Repr.\ $p$}  & {Distribution}  
 			\\[5pt]
 			\hline		\hline
 			1 & $\pattern{scale = 0.6}{2}{1/1,2/2}{}$ & \multirow{1}{2.5cm}{\tiny{Non-inversions given by \eqref{inv-distr}; \cite[p. 21]{Stanley}}
 		    }
 			&
 			20 & $\pattern{scale = 0.6}{2}{1/1,2/2}{0/0,0/1,0/2,1/1,1/2,2/0,2/1}$ &  Theorem~\ref{triv-thm-20}
 			\\[5pt]
 			\hline
 			3 & $\pattern{scale = 0.6}{2}{1/1,2/2}{0/0,0/1,1/2}$ & 
			Conjecture~\ref{conj-nr-3}
 			& 
 			21 & $\pattern{scale = 0.6}{2}{1/1,2/2}{0/1,1/2,0/0,2/0,2/2}$ & Theorem~\ref{triv-thm-21}
 			\\[5pt]
 			\hline
 			5 & $\pattern{scale = 0.6}{2}{1/1,2/2}{0/0,0/1,0/2}$ & Theorem~\ref{triv-thm-5}
 			&
 			22 & $\pattern{scale = 0.6}{2}{1/1,2/2}{0/1,1/2,0/0,2/0,2/2,1/1}$ & Theorem~\ref{triv-thm-22}
 			\\[5pt] 
 			\hline
 		    8 & $\pattern{scale = 0.6}{2}{1/1,2/2}{0/0,0/1,1/0,1/1}$ &  \multirow{3}{2.5cm}{ Theorem~\ref{dis-patterns-8-9} \tiny{Unsigned Stirling numbers of the first kind, \cite[A132393]{OEIS}}
 		    } 
 		    & 
 			27 & $\pattern{scale=0.6}{2}{1/1,2/2}{0/1,2/0,2/2,1/0,1/1,0/2}$    & Theorem~\ref{thm-pat-27}
 		    \\[5pt] 
 		    \cline{4-6}
 			9 & $\pattern{scale = 0.6}{2}{1/1,2/2}{0/1,1/1,1/2,2/1}$ &
 			&
			28 & $\pattern{scale=0.6}{2}{1/1,2/2}{0/1,1/2,0/0,2/2,1/0,2/1}$    &  Theorem~\ref{thm-pat-28}
 			\\[5pt]
 			\hline
 			10 & $\pattern{scale = 0.6}{2}{1/1,2/2}{0/0,0/1,0/2,2/0,2/1,2/2}$ & Theorem~\ref{triv-thm-10}
 			&
 			30 & $\pattern{scale=0.6}{2}{1/1,2/2}{0/1,1/2,2/0,1/0,1/1,2/1,0/2}$ & Theorem~\ref{thm-pat-30}
 			\\[5pt]
 			\hline
 			11 & $\pattern{scale = 0.6}{2}{1/1,2/2}{0/0,0/1,0/2,1/0,1/1,1/2,2/0,2/1,2/2 }$ & Theorem~\ref{triv-thm-11}
 			&
 			33 & $\pattern{scale=0.6}{2}{1/1,2/2}{0/1,1/2,2/0,1/0,0/2,2/1}$   & Theorem~\ref{thm-pat-33}
 			\\[5pt]
 			\hline
 			12 & $\pattern{scale = 0.6}{2}{1/1,2/2}{0/0,0/1,0/2,1/0,2/0}$ & Theorem~\ref{triv-thm-12}
 			&
 			34 & $\pattern{scale=0.6}{2}{1/1,2/2}{0/1,1/2,0/0,2/2,1/0,1/1,2/1}$ & Theorem~\ref{thm-pat-34}
 			\\[5pt]
 			\hline
 			13 & $\pattern{scale = 0.6}{2}{1/1,2/2}{0/0,0/1,0/2,1/0,1/2,2/0,2/1,2/2}$ &  Theorem~\ref{triv-thm-13}
 			&
 			36 & $\pattern{scale=0.6}{2}{1/1,2/2}{0/1,1/2,0/0,1/0,1/1,2/1}$   &  Theorem~\ref{dis-patterns-36}
 			\\[5pt]
 			\hline
 			14 & $\pattern{scale = 0.6}{2}{1/1,2/2}{0/1,1/1,1/2,1/0,1/2,2/1}$ & 
			 \multirow{3}{2.5cm}{Theorem \ref{dis-patterns-14-15}  \footnotesize{small descents, \cite[A123513]{OEIS}}}
 			&
 			45 & $\pattern{scale=0.6}{2}{1/1,2/2}{0/1,1/2,1/0,1/1,2/1,0/2}$     &  Theorem~\ref{dis-patterns-45}
 			\\[5pt]
 			\cline{4-6}
 			15 & $\pattern{scale = 0.6}{2}{1/1,2/2}{0/1,0/2,1/0,1/1,1/2}$ &  
 			&
 			55 & $\pattern{scale = 0.6}{2}{1/1,2/2}{0/1,1/2,0/0,2/0,1/1,2/1}$ &  Theorem~\ref{thm-pat-55}
 			\\[5pt]
 			\hline
 			16 & $\pattern{scale = 0.6}{2}{1/1,2/2}{0/1,2/0,1/0,0/2}$ & Theorem~\ref{thm-pat-16}
 			&
 			56 & $\pattern{scale = 0.6}{2}{1/1,2/2}{0/1,1/2,0/0,2/2,1/1,2/1}$ & Theorem~\ref{thm-pat-56}
 			\\[5pt]
			\hline
 			17 & $\pattern{scale = 0.6}{2}{1/1,2/2}{0/1,1/2,0/0,2/0,1/0,0/2,2/1}$ & Theorem~\ref{thm-pat-17}
 			&
 			63 & $\pattern{scale = 0.6}{2}{1/1,2/2}{0/1,1/2,0/0,2/1,2/0}$ & Theorem~\ref{thm-pat-63}
 			\\[5pt]
 			\cline{1-3}
 			18 & $\pattern{scale = 0.6}{2}{1/1,2/2}{0/0,0/1,0/2,1/2,2/0,2/2}$ &  Theorem~\ref{triv-thm-18}
 			&
 			64 & $\pattern{scale = 0.6}{2}{1/1,2/2}{0/1,1/2,2/0,0/2,1/1}$ & Theorem~\ref{thm-pat-64}
 			\\[5pt]
 			\cline{1-3}
 			19 & $\pattern{scale = 0.6}{2}{1/1,2/2}{0/1,0/2,1/1,1/2,2/0,2/2}$ & Theorem~\ref{triv-thm-19}
 			&
 			65 & $\pattern{scale = 0.6}{2}{1/1,2/2}{0/1,1/0,0/0,1/1,2/2}$ & Theorem~\ref{thm-pat-65}
 			\\[5pt]
 			\hline
 		\end{tabular}
\end{center} 	}
 	\caption{Known or conjectured distributions of mesh patterns of length 2. Pattern's number comes from \cite{Hilmarsson2015Wilf}. Absence of a horizontal line indicates  equidistributions.}\label{tab-1}
\end{table}

\begin{table}[htbp]
{
		\renewcommand{\arraystretch}{1.3}
	\begin{center}
	\begin{tabular}{|c|c|c|c||c|c|c|}
		\hline
		&  	{Nr.\ } & {Repr.\ $p$}  & {Ref.}  &  {Nr.\ } & {Repr.\ $p$}  & {Ref.}  
		\\[6pt]
		\hline \hline
		{proved} 
		& 48 & $\pattern{scale = 0.6}{2}{1/1,2/2}{0/1,1/2,0/0,2/1,2/2}$&\multirow{2}{*}{Theorem~\ref{eqdis-patterns-48-49}}&&&
		 \\[6pt]	\cline{2-3} 
     {equidistributions}   & 49 &	$\pattern{scale = 0.6}{2}{1/1,2/2}{0/1,1/2,0/0,1/1,2/0}$   &	 &&&	\\[6pt]	\hline
   	 & 	23 &	$\pattern{scale = 0.6}{2}{1/1,2/2}{0/0,0/2,1/0,1/1,1/2}$  & \multirow{2}{*}{N/A}  &53 &	$\pattern{scale = 0.6}{2}{1/1,2/2}{0/1,1/2,0/0,2/1}$  & \multirow{2}{*}{N/A}
		\\[6pt]
	\multirow{2}{*}{conjectured} & 	24 &	$\pattern{scale = 0.6}{2}{1/1,2/2}{0/0,0/1,1/0,1/1,1/2}$ &&	54 &  $\pattern{scale = 0.6}{2}{1/1,2/2}{0/1,0/0,1/1,2/2}$   &			\\[6pt] \cline{2-7} 
	\multirow{3}{*}{equidistributions} &	 \multirow{2}{*}{48} & \multirow{2}{*}{$\pattern{scale = 0.6}{2}{1/1,2/2}{0/1,1/2,0/0,2/1,2/2}$}	  & 	
	&	57 & $\pattern{scale = 0.6}{2}{1/1,2/2}{0/1,1/2,1/1,2/0}$	&	 \multirow{2}{*}{N/A}	\\[6pt]		
	 	& \multirow{2}{*}{49} &   \multirow{2}{*}{$\pattern{scale = 0.6}{2}{1/1,2/2}{0/1,1/2,0/0,1/1,2/0}$}   &	 \multirow{2}{*}{N/A}
		&	58 &  $\pattern{scale = 0.6}{2}{1/1,2/2}{0/1,1/0,1/1,2/2}$ &	 	\\[6pt]	\cline{5-7}
	&\multirow{2}{*}{50} &  \multirow{2}{*}{$\pattern{scale = 0.6}{2}{1/1,2/2}{0/1,1/2,0/0,1/1,2/2}$} & & 61 & $\pattern{scale = 0.6}{2}{1/1,2/2}{0/1,1/2,0/0,2/0}$	&  \multirow{2}{*}{N/A}		\\[6pt]
		&  && &	62 & $\pattern{scale = 0.6}{2}{1/1,2/2}{0/1,1/0,0/0,2/2}$ &		\\[6pt]	\hline
	\end{tabular}
	\end{center}
}
	\caption{Equidistributions for which enumeration is unknown.}
 	\label{tab-2}
\end{table}
 
 Let $S_n$ be the set of all permutations of length $n$, which we call $n$-permutations. For example, $S_3=\{123, 132, 213, 231, 312, 321\}$. 
 For a pattern $p$ and a permutation $\pi$, we let $p(\pi)$ denote the number of occurrences of $p$ in $\pi$. Also, let $S_n(p)$ denote the set of all permutations of length $n$ avoiding $p$ and $S(p)=\cup_{n\geq 0}S_n(p)$. Finally, throughout this paper, we let ``g.f.'' stand for ``generating function'' and  let $$F(x)=\sum_{n\geq 0}n!x^n.$$
 
 Our main enumerative method is via deriving a functional equation for the generating function in question, and solving it; we use recurrence relations in the remaining cases. We illustrate our typical approach in detail on  first finding avoidance, and then distribution of the pattern \pattern{scale=0.8}{1}{1/1}{0/0,1/1}, which is equivalent to the pattern
 \pattern{scale=0.8}{1}{1/1}{0/1,1/0} via applying the reverse operation. Occurrences of the pattern \pattern{scale=0.8}{1}{1/1}{0/1,1/0} are known as {\em strong fixed points}, and the case of avoidance was already given in \cite{Baril}; also, see \cite[A052186]{OEIS}. However, the distribution of strong fixed points seems to be a new result.
The following theorem is  used frequently throughout this paper.

\begin{thm}\label{thm-length-1} Let
 $$F(x,q)=\sum_{n\geq 0}x^n\sum_{\pi\in S_n}q^{\pattern{scale=0.5}{1}{1/1}{0/1,1/0}(\pi)}=\sum_{n\geq 0}x^n\sum_{\pi\in S_n}q^{\pattern{scale=0.5}{1}{1/1}{0/0,1/1}(\pi)},$$ and $A(x)$ be the g.f. for $S(\pattern{scale=0.5}{1}{1/1}{0/1,1/0}\ )=S(\pattern{scale=0.5}{1}{1/1}{0/0,1/1}\ )$. Then, 
 $$A(x)=\frac{F(x)}{1+xF(x)};\ \ \ \ \ F(x,q)=\frac{F(x)}{1+x(1-q)F(x)}.$$
\end{thm}

\begin{proof} We first consider the pattern \pattern{scale=0.8}{1}{1/1}{0/1,1/0}. We claim that 
\begin{equation}\label{av-length-1}A(x)+xA(x)F(x)=F(x).\end{equation}
Indeed, each permutation $\pi$, which is counted by the $F(x)$ on the right hand side of \eqref{av-length-1}, either avoids \pattern{scale=0.8}{1}{1/1}{0/1,1/0} (and thus is counted by the $A(x)$ on the left hand side of \eqref{av-length-1}), or contains at least one occurrence of \pattern{scale=0.8}{1}{1/1}{0/1,1/0}. Consider the leftmost occurrence of \pattern{scale=0.8}{1}{1/1}{0/1,1/0}, an element $a$, which corresponds to the $x$ in $xA(x)F(x)$. To the left of $a$, we must have a \pattern{scale=0.8}{1}{1/1}{0/1,1/0}-avoiding permutation $\pi'$ formed by the smallest elements of $\pi$ and counted by $A(x)$. To the right of $a$, we can have any permutation $\pi''$ formed by the largest elements of $\pi$ and counted by $F(x)$. 
Hence, we obtain the equation \eqref{av-length-1} due to the independence of the choices of $\pi'$ and $\pi''$. 
By solving \eqref{av-length-1}, the desired formula for  $A(x)$ follows.

Similarly, we can derive 
\begin{equation}\label{length-1-formula}A(x)+xqA(x)F(x,q)=F(x,q),\end{equation} where the $q$ in $xqA(x)F(x,q)$ indicates that the element $a$ contributes to occurrences of \pattern{scale=0.8}{1}{1/1}{0/1,1/0}. 
By substituting $A(x)$ found above into (\ref{length-1-formula}), and solving for $F(x,q)$, we obtain the desired result. 
\end{proof}

\section{``Trivial'' distributions}

By a ``trivial'' distribution we mean the situation when either the pattern in question can occur at most once and its avoidance was given in \cite{Hilmarsson2015Wilf}, or pattern's occurrences can easily be understood from the shape of the pattern. There are 10 such patterns:
\begin{center}
\begin{tabular}{rrrr}
Nr.\  5  = $\pattern{scale = 0.6}{2}{1/1,2/2}{0/0,0/1,0/2}$; &
Nr.\ 10 = $\pattern{scale = 0.6}{2}{1/1,2/2}{0/0,0/1,0/2,2/0,2/1,2/2}$; & 
Nr.\ 11 = $\pattern{scale = 0.6}{2}{1/1,2/2}{0/0,0/1,0/2,1/0,1/1,1/2,2/0,2/1,2/2 }$;&
Nr.\ 12 = $\pattern{scale = 0.6}{2}{1/1,2/2}{0/0,0/1,0/2,1/0,2/0}$; \\[5pt]
Nr.\ 13 = $\pattern{scale = 0.6}{2}{1/1,2/2}{0/0,0/1,0/2,1/0,1/2,2/0,2/1,2/2}$; &
Nr.\ 18 = $\pattern{scale = 0.6}{2}{1/1,2/2}{0/0,0/1,0/2,1/2,2/0,2/2}$; &
Nr.\ 19 = $\pattern{scale = 0.6}{2}{1/1,2/2}{0/1,0/2,1/1,1/2,2/0,2/2}$; &
Nr.\ 20 = $\pattern{scale = 0.6}{2}{1/1,2/2}{0/0,0/1,0/2,1/1,1/2,2/0,2/1}$;\\[5pt]
Nr.\ 21 = $\pattern{scale = 0.6}{2}{1/1,2/2}{0/1,1/2,0/0,2/0,2/2}$; &
Nr.\ 22 = $\pattern{scale = 0.6}{2}{1/1,2/2}{0/1,1/2,0/0,2/0,2/2,1/1}$. &&
\end{tabular}
\end{center}
To be self-contained, and to allow proper references in Table~\ref{tab-1},  in this section we state the distribution results for the 10 patterns as separate theorems. 

The following theorem is easy to see, and we omit its proof. 

\begin{thm}\label{triv-thm-5} For permutations of length $n\geq 2$ beginning with an element $k$ ($1\leq k\leq n$), there are $(n-1)!$ permutations containing $n-k$ occurrences of the pattern \pattern{scale = 0.6}{2}{1/1,2/2}{0/0,0/1,0/2}. 
\end{thm}

\begin{thm}\label{triv-thm-10} For permutations of length $n\geq 2$, there are $n!/2$ permutations avoiding the pattern $\pattern{scale = 0.6}{2}{1/1,2/2}{0/0,0/1,0/2,2/0,2/1,2/2}$, and $n!/2$ permutations containing it exactly once. 
\end{thm}

\begin{proof} Clearly, in exactly half of $n$-permutations the leftmost element is less than the rightmost element, in which case the elements form the only possible occurrence of the pattern.\end{proof}

The following theorem is trivial. 
\begin{thm}\label{triv-thm-11} The only permutation that contains the pattern $\pattern{scale = 0.6}{2}{1/1,2/2}{0/0,0/1,0/2,1/0,1/1,1/2,2/0,2/1,2/2 }$
(once)  is $12$.
\end{thm}

\begin{thm}\label{triv-thm-12} For permutations of length $n\geq 2$, there are $(n-1)!$ permutations containing $(n-1)!$ occurrences of the pattern $p = \pattern{scale = 0.6}{2}{1/1,2/2}{0/0,0/1,0/2,1/0,2/0}$, and $n!-(n-1)!$ permutations avoiding it.    
\end{thm}

\begin{proof} Any occurrence of $p$ must start with the element 1 placed in position 1, so there are $n!-(n-1)!$ permutations avoiding $p$. In the remaining $(n-1)!$ permutations, each element, together with the element 1, contributes an occurrence of $p$. \end{proof}

\begin{thm}\label{triv-thm-13} For permutations of length $n\geq 2$, there are $(n-2)!$ permutations containing exactly one occurrence of the pattern $p = \pattern{scale = 0.6}{2}{1/1,2/2}{0/0,0/1,0/2,1/0,1/2,2/0,2/1,2/2}$, and $n!-(n-2)!$ permutations avoiding it.  
\end{thm}

\begin{proof} The only possible occurrence of $p$ can be formed by the element 1 in the leftmost position, and the element $n$ in the rightmost position, which gives the desired result.\end{proof}

\begin{thm}\label{triv-thm-18} For permutations of length $n\geq 2$, there are $$n!-\sum_{i=1}^{n-1}\frac{(n-1)!}{i}$$ permutations avoiding  the pattern $p = \pattern{scale = 0.6}{2}{1/1,2/2}{0/0,0/1,0/2,1/2,2/0,2/2}$, while the remaining permutations contain exactly one occurrence of  the pattern $p$. 
\end{thm}

\begin{proof} Clearly, any permutation can contain at most one occurrence of $p$ formed by the leftmost and the largest elements. The theorem now follows from the avoidance result in \cite[Prop. 24]{Hilmarsson2015Wilf}. \end{proof}

\begin{thm}\label{triv-thm-19} For permutations of length $n\geq 2$, there are  $$n!-\sum_{i=0}^{n-2}i!(n-i-1)!$$ permutations avoiding  the pattern $p = \pattern{scale = 0.6}{2}{1/1,2/2}{0/1,0/2,1/1,1/2,2/0,2/2}$, while the remaining permutations contain exactly one occurrence of $p$.   
\end{thm}

\begin{proof} Clearly, any permutation can contain at most one occurrence of $p$ formed by the element $n$ and the largest element to the left of it. The theorem now follows from the easy to see avoidance result in \cite[Prop. 26]{Hilmarsson2015Wilf}.\end{proof}

\begin{thm}\label{triv-thm-20} For permutations of length $n\geq 2$,  there are $$n!-\sum_{i=1}^{n-1}(i-1)!(n-i-1)!$$ permutations avoiding  the pattern $p =\pattern{scale = 0.6}{2}{1/1,2/2}{0/0,0/1,0/2,1/1,1/2,2/0,2/1}$, while the remaining permutations contain exactly one occurrence of $p$.   
 \end{thm}

\begin{proof} Clearly, any permutation can contain at most one occurrence of $p$ formed by the leftmost element $x$ and the element $x+1$. The theorem now follows from the easy-to-see avoidance result in \cite[Prop. 27]{Hilmarsson2015Wilf}. \end{proof}

\begin{thm}\label{triv-thm-21} For permutations of length $n\geq 2$, there  are $$n!-\sum_{i=1}^{n-1}\sum_{\ell=1}^{i}\ell!(i-\ell)!(n-i-\ell)!$$ permutations avoiding the pattern $p=\pattern{scale = 0.6}{2}{1/1,2/2}{0/1,1/2,0/0,2/0,2/2}$, while the remaining permutations contain exactly one occurrence of $p$.   
\end{thm}

\begin{proof} It is not difficult to see that any permutation can contain at most one occurrence $ab$ of $p$. Indeed, the bottom right shaded box in $p$ guarantees that $p$ cannot occur to the left of $a$, while the presence of $a$ and $b$ guarantees that no other occurrence of $p$ can involve $a$, or any other elements to the right of $a$. The theorem now follows from the avoidance result in \cite[Prop. 28]{Hilmarsson2015Wilf}.\end{proof}

\begin{thm}\label{triv-thm-22} For permutations of length $n\geq 2$, there are $$n!-\sum_{i=0}^{n-2}\sum_{\ell=0}^{i}\ell!(i-\ell)!(n-2-i)!$$ permutations avoiding the pattern $p = \pattern{scale = 0.6}{2}{1/1,2/2}{0/1,1/2,0/0,2/0,2/2,1/1}$, while the remaining permutations contain exactly one occurrence of $p$.    
\end{thm}

\begin{proof} By the same reasons as in the proof of Theorem~\ref{triv-thm-21} we conclude that $p$ cannot occur more than once. The theorem now follows from the avoidance result in \cite[Prop. 29]{Hilmarsson2015Wilf}. \end{proof}

\section{The generating functions method}
In this section, we use the approach similar to, but in several cases (much) more involved than, the proof of Theorem~\ref{thm-length-1} to find the distribution and, whenever appropriate, avoidance for the following 12 patterns: 
\begin{center}
\begin{tabular}{rrrr}
Nr.\ 16 = $\pattern{scale = 0.6}{2}{1/1,2/2}{0/1,2/0,1/0,0/2}$; &
Nr.\ 17 = $\pattern{scale = 0.6}{2}{1/1,2/2}{0/1,1/2,0/0,2/0,1/0,0/2,2/1}$; &
Nr.\ 27 = $\pattern{scale=0.6}{2}{1/1,2/2}{0/1,2/0,2/2,1/0,1/1,0/2}$; &
Nr.\ 28 = $\pattern{scale=0.6}{2}{1/1,2/2}{0/1,1/2,0/0,2/2,1/0,2/1}$;\\[5pt]
Nr.\ 30 = $\pattern{scale=0.6}{2}{1/1,2/2}{0/1,1/2,2/0,1/0,1/1,2/1,0/2}$; &
Nr.\ 33 = $\pattern{scale=0.6}{2}{1/1,2/2}{0/1,1/2,2/0,1/0,0/2,2/1}$;  & 
Nr.\ 34 = $\pattern{scale=0.6}{2}{1/1,2/2}{0/1,1/2,0/0,2/2,1/0,1/1,2/1}$;  &
Nr.\ 55 = $\pattern{scale = 0.6}{2}{1/1,2/2}{0/1,1/2,0/0,2/0,1/1,2/1}$;\\[5pt]
Nr.\ 56 = $\pattern{scale = 0.6}{2}{1/1,2/2}{0/1,1/2,0/0,2/2,1/1,2/1}$;  &
Nr.\ 63 = $\pattern{scale = 0.6}{2}{1/1,2/2}{0/1,1/2,0/0,2/1,2/0}$;  & 
Nr.\ 64 = $\pattern{scale = 0.6}{2}{1/1,2/2}{0/1,1/2,2/0,0/2,1/1}$;  & 
Nr.\ 65 = $\pattern{scale = 0.6}{2}{1/1,2/2}{0/1,1/0,0/0,1/1,2/2}$.
\end{tabular}
\end{center}

\subsection{Distribution of the pattern Nr.\ 16}  
Our next theorem establishes the avoidance and  distribution of the pattern Nr.\ 16 = $\pattern{scale = 0.6}{2}{1/1,2/2}{0/1,2/0,1/0,0/2}$.

\begin{thm}\label{thm-pat-16}
 Let $p=\pattern{scale = 0.4}{2}{1/1,2/2}{0/1,2/0,1/0,0/2}$, $F(x,q)=\sum_{n\geq 0}x^n\sum_{\pi\in S_n}q^{p(\pi)}$, and $A(x)$ be the g.f. for $S(p)$. 
 Then, 
 $$A(x)=\frac{(1+x)F(x)}{1+xF(x)};\qquad  F(x,q)=\sum_{i\geq 0}q^{i\choose 2}x^i\prod_{j=0}^{i}\frac{F(q^jx)}{1+xq^jF(q^jx)}.$$
\end{thm}

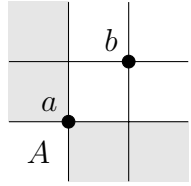
\begin{figure}
\begin{center}
	\begin{tikzpicture}[scale=0.8, baseline=(current bounding box.center)]
	\foreach \x/\y in {0/1,0/2,1/0,2/0}		    
		\fill[gray!20] (\x,\y) rectangle +(1,1);
	\draw (0.01,0.01) grid (2+0.99,2+0.99);
	\filldraw (1,1) circle (3pt) node[above left] {$a$};
	\filldraw (2,2) circle (3pt) node[above left] {$b$};
	\node  at (0.5,0.5) {$A$};
	\end{tikzpicture}
\end{center}
\caption{Related to the proof of Theorem~\ref{thm-pat-16}}\label{pic-thm-pat-16}
\end{figure}

\begin{proof} 
We claim that 
\begin{equation}\label{av-pattern-16} 
A(x) + x\big(F(x)-1\big)\frac{F(x)}{1+xF(x)}=F(x). 
\end{equation}
Indeed, each permutation $\pi$, which is counted by the right hand side in \eqref{av-pattern-16}, either avoids $p$ (which is counted by the $A(x)$ term in \eqref{av-pattern-16}), or contains at least one occurrence of $p$. Among all such occurrences, pick an occurrence $ab$ with the \emph{leftmost} possible $a$ as shown in Fig.~\ref{pic-thm-pat-16}. Referring to this figure, we note that the permutation $A$ must be \pattern{scale=0.8}{1}{1/1}{0/1,1/0}-avoiding, or else, $a$ is not the leftmost possible. By Theorem~\ref{thm-length-1}, this explains the term $\cfrac{F(x)}{1+xF(x)}$ in \eqref{av-pattern-16} enumerating such $A$'s. Further, since $ab$ is an occurrence of $p$, to the right of $a$ in $\pi$ we must have a non-empty permutation, which can be any, and such permutations are counted by $F(x)-1$. Finally, $a$ contributes the factor of $x$. This completes the proof of~\eqref{av-pattern-16} and gives the desired formula for $A(x)$:
$$A(x)=\frac{(1+x)F(x)}{1+xF(x)}.$$

For the distribution, we have the following functional equation:
\begin{equation}\label{dis-pattern-16} 
A(x) + x\big(F(qx,q)-1\big)\frac{F(x)}{1+xF(x)}=F(x,q). 
\end{equation}
The proof of \eqref{dis-pattern-16} is essentially the same as that in the avoidance case. The only term that may warrant  an explanation is $F(qx,q)$. Observe that any element to the right of $a$, together with $a$, forms an occurrence of $p$. Hence, the $x$ in $F(x,q)$, which counts permutations to the right of $a$ with respect to occurrences of $p$, should be substituted by $xq$. 
This completes the proof of \eqref{dis-pattern-16}.
By substituting the formula for $A(x)$ found in the above paragraph into \eqref{dis-pattern-16}, we get that
$$F(x,q)=\frac{F(x)}{1+xF(x)}\big(1+xF(qx,q)\big).$$
Iterating the above formula repeatedly, we obtain the desired result.
\end{proof}

\subsection{Distribution of the pattern Nr.\ 17} 
The avoidance for the pattern Nr.\ 17 = $\pattern{scale = 0.6}{2}{1/1,2/2}{0/1,1/2,0/0,2/0,1/0,0/2,2/1}$ is given by~\cite[Prop. 25]{Hilmarsson2015Wilf}. Our next theorem establishes the distribution of the pattern Nr.\ 17 = $\pattern{scale = 0.6}{2}{1/1,2/2}{0/1,1/2,0/0,2/0,1/0,0/2,2/1}$.

\begin{thm}\label{thm-pat-17}
 Let $p=\pattern{scale = 0.6}{2}{1/1,2/2}{0/1,1/2,0/0,2/0,1/0,0/2,2/1}$, $F(x,q)=\sum_{n\geq 0}x^n\sum_{\pi\in S_n}q^{p(\pi)}$, and $A(x)$ be the g.f. for $S(p)$. Then, $$F(x,q)=\left(1-x + \frac{x}{1+x(1-q)F(x)}\right)F(x).$$
\end{thm}


\begin{proof} Observe that any occurrence of $p$ must start with the element 1 in the first position, and the number of occurrences is then given by the number of occurrences of the pattern \pattern{scale=0.8}{1}{1/1}{0/1,1/0} to the right of 1. Thus, each permutation either 
\begin{itemize}
\item avoids $p$, so it does not start with $1$, and such permutations are counted by $F(x)-xF(x)$, or 
\item begins with 1 corresponding to $x$, and the distribution of $p$ is given by the distribution of the pattern \pattern{scale=0.8}{1}{1/1}{0/1,1/0} in Theorem~\ref{thm-length-1}.
\end{itemize}
Hence, we obtain that 
\begin{align*}
F(x,q)=	F(x)-xF(x)  +  \frac{x F(x)}{1+x(1-q)F(x)}, 
\end{align*}
which completes the proof.
\end{proof}

\subsection{Distribution of the pattern Nr.\ 27} 
Our next theorem establishes the avoidance and  distribution of the pattern Nr.\ 27 = $\pattern{scale=0.6}{2}{1/1,2/2}{0/1,2/0,2/2,1/0,1/1,0/2}$.

\begin{thm}\label{thm-pat-27}
 Let $p=\pattern{scale=0.6}{2}{1/1,2/2}{0/1,2/0,2/2,1/0,1/1,0/2}$, $F(x,q)=\sum_{n\geq 0}x^n\sum_{\pi\in S_n}q^{p(\pi)}$, and $A(x)$ be the g.f. for $S(p)$. Then, $$A(x)=F(x)-\frac{x^2F(x)^3}{1+xF(x)};\ \ \ \ \ F(x,q)=F(x)-\frac{(1-q)x^2F^3(x)}{1+x(1-q)F(x)}.$$
\end{thm}

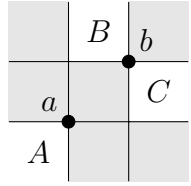
\begin{figure}
\begin{center}
	\begin{tikzpicture}[scale=0.8, baseline=(current bounding box.center)]
	\foreach \x/\y in {0/1,0/2,1/0,1/1,2/0,2/2}		    
	\fill[gray!20] (\x,\y) rectangle +(1,1);
	\draw (0.01,0.01) grid (2+0.99,2+0.99);
	\filldraw (1,1) circle (3pt) node[above left] {$a$};
	\filldraw (2,2) circle (3pt) node[above right] {$b$};
	\node  at (0.5,0.5) {$A$};
	\node  at (1.5,2.5) {$B$};
	\node  at (2.5,1.5) {$C$};
	\end{tikzpicture}
\end{center}
\caption{Related to the proof of Theorem~\ref{thm-pat-27}}\label{av-pic-pat-27}
\end{figure}

\begin{proof} We claim that 
\begin{equation}\label{av-pattern-27} A(x)+x^2\frac{F(x)}{1+xF(x)}F^2(x)=F(x).\end{equation}
Indeed, permutations avoiding $p$ are counted by $A(x)$ in \eqref{av-pattern-27}. To complete the proof of \eqref{av-pattern-27}, we need to show that permutations containing occurrences of $p$ are counted by $x^2\dfrac{F(x)}{1+xF(x)}F^2(x)$. Let $ab$ be the occurrence of $p$ in such a permutation $\pi$ with the property that both $a$ and $b$ are  the \emph{leftmost} possible; see Fig.~\ref{av-pic-pat-27}. The factor of $x^2$ is given by $ab$.

Referring to Fig.~\ref{av-pic-pat-27}, we note that $A$ can be any permutation because we will never get a contradiction with $ab$ being the leftmost occurrence of $p$ in $\pi$. Indeed, any occurrence $a'b'$ of $p$ in $A$ is not an occurrence of $p$ in $\pi$ because the elements $a$ and $b$ will be in the shaded area North-East of $b'$. Also, because of the presence of $a$, there is no occurrence $a'b'$ of $p$ which start in $A$ and end in $B$ or $C$ ($a$ would be in the shaded area between $a'$ and $b'$). Thus, $A$ contributes a factor of $F(x)$. So does $C$ since there are no restrictions for it. Finally, the only restriction for $B$ is that it must be \pattern{scale=0.8}{1}{1/1}{0/0,1/1}-avoiding because $b$ is the leftmost possible. Theorem~\ref{thm-length-1} can now be applied to justify the factor of $\cfrac{F(x)}{1+xF(x)}$ contributed by $C$. Solving \eqref{av-pattern-27} for $A(x)$ we obtain the desired result.

The formula for $F(x,q)$ follows from the following relation to be proved:
\begin{equation}\label{dis-pattern-27} 
A(x)+qx^2F(x)\frac{F(x)}{1+xF(x)}\sum_{i\geq 0}(qx)^i\left(\frac{F(x)}{1+xF(x)}\right)^{i+1}=F(x,q).
\end{equation}
We can proceed like in the case of avoidance using Fig.~\ref{av-pic-pat-27}, so that, in particular, in $qx^2F(x)\cfrac{F(x)}{1+xF(x)}$, $qx^2$ is the contribution of the occurrence $ab$ of $p$, $F(x)$ is given by $A$, and the rest is given by $B$. 

Next, we observe that an occurrence of $p$ inside $C$, if any, is not an occurrence of $p$ in $\pi$ because of the element $b$. However, each occurrence of  \pattern{scale=0.8}{1}{1/1}{0/0,1/1} in $C$, together with $a$, gives an occurrence of $p$ in $\pi$. Thus, the structure in Fig.~\ref{av-pic-pat-27} can be refined to that in Fig.~\ref{dis-pic-pat-27}, where $b_1, b_2,\ldots$ may, or may not,  exist. The index $i$ in the sum in (\ref{dis-pattern-27}) is responsible for the number of $b_j$'s. Each $b_j$ contributes $qx$, and each $B_j$, being \pattern{scale=0.8}{1}{1/1}{0/0,1/1}-avoiding, contributes $F(x)/(1+xF(x))$ in (\ref{dis-pattern-27}); note that the number of $B_j$'s is $i+1$.

We observe that the second term in (\ref{dis-pattern-27}) is
$$\frac{xF^2(x)}{1+xF(x)}\left(\sum_{i\geq 0}\left(\frac{qxF(x)}{1+xF(x)}\right)^i-1\right)=\frac{qx^2F^3(x)}{(1+xF(x))(1+x(1-q)F(x))}.$$
Substituting the formula for $A(x)$ found above into \eqref{dis-pattern-27} we get the desired expression of $F(x,q)$.
\end{proof}

\begin{figure}[!ht]
\begin{center}
	\begin{tikzpicture}[scale=0.8, baseline=(current bounding box.center)]
	\foreach \x/\y in {0/1,0/2,0/3,0/4,1/0,1/1,1/2,1/3,2/0,2/1,2/2,2/4,3/0,3/1,3/3,3/4,4/0,4/2,4/3,4/4}		    
	\fill[gray!20] (\x,\y) rectangle +(1,1);
	\draw (0.01,0.01) grid (4+0.99,4+0.99);
	\filldraw (1,1) circle (3pt) node[above left] {$a$};
	\filldraw (2,4) circle (3pt) node[above right] {$b$};
	\filldraw (3,3) circle (3pt) node[above right] {$b_1$};
	\filldraw (4,2) circle (3pt) node[above right] {$b_2$};
	\node  at (0.5,0.5) {$A$};
	\node  at (1.5,4.5) {$B$};
	\node  at (2.5,3.5) {$B_0$};
	\node  at (3.5,2.5) {$B_1$};
	\node  at (4.5,1.5) {$\ddots$};
	\end{tikzpicture}
\end{center}
\caption{Related to the proof of Theorem~\ref{thm-pat-27}}\label{dis-pic-pat-27}
\end{figure}
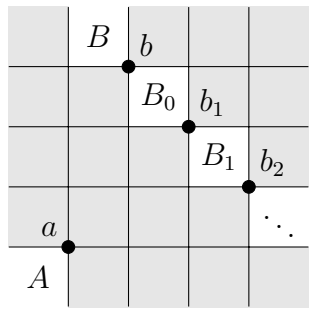

\subsection{Distribution of the pattern Nr.\ 28}
Our next theorem establishes the avoidance and  distribution of the pattern  
Nr.\ 28 = $\pattern{scale=0.6}{2}{1/1,2/2}{0/1,1/2,0/0,2/2,1/0,2/1}$.

\begin{thm}\label{thm-pat-28}
	Let $p=\pattern{scale=0.6}{2}{1/1,2/2}{0/1,1/2,0/0,2/2,1/0,2/1}$, $F(x,q)=\sum_{n\geq 0}x^n\sum_{\pi\in S_n}q^{p(\pi)}$, and $A(x)$ be the g.f. for $S(p)$. Then, 
	\begin{align*}
		A(x)=\frac{F(x)}{1+x^2F^2(x)};
		\ \ \ \ \ 
		F(x,q)=\frac{F(x)}{1+x^2(1-q)F^2(x)}.
	\end{align*}
\end{thm}

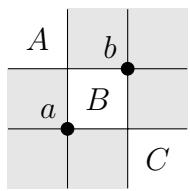
\begin{figure}[!ht]
\begin{center}
	\begin{tikzpicture}[scale=0.8, baseline=(current bounding box.center)]
	\foreach \x/\y in {0/0,0/1,1/0,1/2,2/1,2/2}		    
	\fill[gray!20] (\x,\y) rectangle +(1,1);
	\draw (0.01,0.01) grid (2+0.99,2+0.99);
	\filldraw (1,1) circle (3pt) node[above left] {$a$};
	\filldraw (2,2) circle (3pt) node[above left] {$b$};
	\node  at (0.5,2.5) {$A$};
	\node  at (1.5,1.5) {$B$};
	\node  at (2.5,0.5) {$C$};
	\end{tikzpicture}
	\caption{Related to the proof of Theorem~\ref{thm-pat-28}}\label{pic-thm-pat-28}
\end{center}
\end{figure}

\begin{proof}
We claim that 
\begin{equation}\label{av-pattern-28} 
A(x) +x^2A(x)F^2(x)=F(x). 
\end{equation}
Indeed, each permutation $\pi$, which is counted by the right hand side in \eqref{av-pattern-28}, either avoids $p$ (which is counted by the $A(x)$ term in \eqref{av-pattern-28}), or  contains at least one occurrence of $p$. Among all such occurrences, pick an occurrence $ab$ with the \emph{leftmost} possible $a$ as shown in Fig.~\ref{pic-thm-pat-28}. Referring to this figure, we note that the permutation $A$ must be $p$-avoiding, or else, $a$ is not the leftmost possible. 
Further,  $B$ and $C$ can be any permutations contributing the factor of $F^2(x)$ in \eqref{av-pattern-28}. Finally, $a$ and $b$ contribute the factor of $x^2$.  Thus,  we complete the proof of   \eqref{av-pattern-28} and  hence give the formula for $A(x)$. 

For the distribution, we have the following functional equation:
\begin{equation}\label{dis-pattern-28} 
A(x) +x^2 q A(x)F(x)F(x,q)=F(x,q). 
\end{equation}
The proof of \eqref{dis-pattern-28} is essentially the same as that in the avoidance case. 
Because of the elements $a$ and $b$, no occurrence of $p$ can be created between $a$ and $b$ and thus the block $B$ contributes $F(x)$.
On the other hand, the block $C$ contributes $F(x,q)$, since all occurrences of $p$ in $C$ are preserved in the whole permutation.
Together with the factor $x^2q$ which corresponds to the elements $a$ and $b$, all the permutations containing occurrences of $p$ are counted by $x^2 q A(x)F(x)F(x,q)$. This completes the proof of 
\eqref{dis-pattern-28} and hence we get the formula for $A(x)$.
Substituting the formula for $A(x)$ found above into \eqref{dis-pattern-28}   we obtain the desired formula for $F(x,q)$.\end{proof}

\subsection{Distribution of the pattern Nr.\ 30} 
Our next theorem establishes the avoidance and the distribution of the pattern  
Nr.\ 30 = $\pattern{scale=0.6}{2}{1/1,2/2}{0/1,1/2,2/0,1/0,1/1,2/1,0/2}$.

\begin{thm}\label{thm-pat-30}
	Let $p=\pattern{scale=0.6}{2}{1/1,2/2}{0/1,1/2,2/0,1/0,1/1,2/1,0/2}$, 
	$F(x,q)=\sum_{n\geq 0}x^n\sum_{\pi\in S_n}q^{p(\pi)}$, and $A(x)$ be the g.f. for $S(p)$. Then, 
	\begin{align*}
	A(x)=\frac{(1+x)F(x)}{1+x+x^2F(x)};
	\ \ \ \ \ 
	F(x,q)=\frac{(1+x-qx)F(x)}{1+(1-q)x+(1-q)x^2F(x)}.
	\end{align*}
\end{thm}

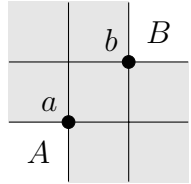
\begin{figure}[!ht]
\begin{center}
	\begin{tikzpicture}[scale=0.8, baseline=(current bounding box.center)]
	\foreach \x/\y in {0/1,0/2,1/0,1/1,1/2,2/0,2/1}		    
	\fill[gray!20] (\x,\y) rectangle +(1,1);
	\draw (0.01,0.01) grid (2+0.99,2+0.99);
	\filldraw (1,1) circle (3pt) node[above left] {$a$};
	\filldraw (2,2) circle (3pt) node[above left] {$b$};
	\node  at (0.5,0.5) {$A$};
	\node  at (2.5,2.5) {$B$};
	\end{tikzpicture}
\caption{Related to the proof of Theorem~\ref{thm-pat-30}}\label{pic-thm-pat-30}
\end{center}
\end{figure}

\begin{proof}
We claim that 
\begin{equation}\label{av-pattern-30} 
A(x) +x^2 \frac{A(x)}{1+x} F(x)=F(x). 
\end{equation}
Indeed, each permutation $\pi$, which is counted by the right hand side in \eqref{av-pattern-30}, either avoids $p$ (which is counted by the $A(x)$ term in \eqref{av-pattern-30}), or  contains at least one occurrence of $p$. Among all such occurrences, pick an occurrence $ab$ with the \emph{leftmost} possible $a$ as shown in Fig.~\ref{pic-thm-pat-30}. Referring to this figure,  the permutation $A$ must be both $p$-avoiding and  \pattern{scale=0.8}{1}{1/1}{0/1,1/0,1/1}-avoiding, or else, $a$ is not the leftmost possible. 
Note that the we must require both, because it is possible to avoid \pattern{scale=0.8}{1}{1/1}{0/1,1/0,1/1} but contain $p$, for example 1243.
Let the g.f. for  permutations in $A$ be $C(x)$.  Then the g.f. for  permutations avoiding  $p$ but containing \pattern{scale=0.8}{1}{1/1}{0/1,1/0,1/1} is $xC(x)$.
Hence, we obtain that 
$C(x)+xC(x)=A(x)$,
which leads to 
$C(x)=\frac{A(x)}{1+x}$.
Further, since $ab$ is an occurrence of $p$, to the right of $a$ in $\pi$ we must have any permutation in $B$. Finally, $a$ and $b$ contribute the factor of $x^2$.  Thus,  we complete the proof of \eqref{av-pattern-30} and  hence give the formula for $A(x)$.

We next consider the distribution. Let $F_1(x,q)$ be the g.f. for the distribution of $p$ on permutations starting with $1$. 
Then, the  g.f. for the distribution of $p$ on permutations starting with $12$ is $qxF_1(x,q)$, since $12$ gives an extra occurrence of $p$ and thus  the  g.f. for the distribution of $p$ on permutations starting with $1$ but not $12$ is $x\left(F(x,q)-F_1(x,q)\right)$.
Therefore, we have that
\begin{align*}
	F_1(x,q) = qx F_1(x,q) + x(F(x,q)-F_1(x,q)),
\end{align*}
and thus
\begin{align}\label{dis-pattern-30-F1} 
F_1(x,q) =\frac{x F(x,q)}{1-qx+x}. 
\end{align}
On the other hand,  we have the following functional equation for $F(x,q)$:
\begin{align}\label{dis-pattern-30} 
A(x) +x^2 q \frac{A(x)}{1+x} \left( qF_1(x,q)+  F(x,q) -F_1(x,q) \right)=F(x,q). 
\end{align}
The proof of \eqref{dis-pattern-30} is essentially the same as that in the avoidance case. 
The block $A$ does not contribute occurrences of $p$, explaining $\frac{A(x)}{1+x}$.
If the block $B$ begins with $1$, then this $1$, together with $b$, forms an occurrence of $p$, contributing  $qF_1(x,q)$.
If $B$ do not start with $1$, then there are no extra occurrences involving $b$ which contribute $F(x,q)-F_1(x,q)$.
Together with the factor $x^2q$ which corresponds to the elements $a$ and $b$, all the permutations containing occurrences of $p$ are counted by $x^2 q \frac{A(x)}{1+x} \left( qF_1(x,q)+  F(x,q) -F_1(x,q) \right)$. 
This completes the proof of \eqref{dis-pattern-30}.
Therefore, we get the formula for $A(x)$.
Substituting \eqref{dis-pattern-30-F1} and the formula for $A(x)$ found in the previous paragraph into \eqref{dis-pattern-30}, we obtain the desired formula for $F(x,q)$. 
\end{proof}

\subsection{Distribution of the pattern Nr.\ 33} 
Our next theorem establishes the avoidance and the distribution of the pattern  
Nr.\ 33 = $\pattern{scale=0.6}{2}{1/1,2/2}{0/1,1/2,2/0,1/0,0/2,2/1}$.

\begin{thm}\label{thm-pat-33}
	Let $p=\pattern{scale=0.6}{2}{1/1,2/2}{0/1,1/2,2/0,1/0,0/2,2/1}$, 
	$F(x,q)=\sum_{n\geq 0}x^n\sum_{\pi\in S_n}q^{p(\pi)}$, and $A(x)$ be the g.f. for $S(p)$. Then, 
	\begin{align*}
	A(x)=\frac{(1+2xF(x))F(x)}{(1+xF(x))^2};
	\ \ \ \ \ 
	F(x,q)=\sum_{i=0}^{\infty}q^{\binom{i}{2}} x^i 
	\left(\frac{F(x)}{1+xF(x)} \right)^{i+1}.
	\end{align*}
\end{thm}

\begin{figure}[!ht]
\begin{center}
	\begin{tikzpicture}[scale=0.8, baseline=(current bounding box.center)]
	\foreach \x/\y in {0/1,0/2,0/3,1/0,1/2,1/3,2/0,2/1,2/3,3/0,3/1,3/2}		 	\fill[gray!20] (\x,\y) rectangle +(1,1);
	\draw (0.01,0.01) grid (3+0.99,3+0.99);
	 \filldraw (1,1) circle (2pt) node[above left]{$a_1$};
	 \filldraw (2,2) circle (2pt) node[above left]{$a_2$};
 	 \filldraw (3,3) circle (2pt) node[above left]{$a_i$};
 	 \node  at (0.5,0.5) {$A_0$};
 	 \node  at (1.5,1.5) {$A_1$};
 	 \node  at (2.5,2.6) {$\iddots$};
 	 \node  at (3.5,3.5) {$A_i$};
	\end{tikzpicture}
	\caption{Related to the proof of Theorem~\ref{thm-pat-33}}\label{pic-thm-pat-33}
\end{center}
\end{figure}
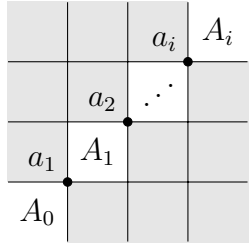

\begin{proof}
For any permutation, consider its decomposition given by the occurrences of \pattern{scale=0.8}{1}{1/1}{0/1,1/0}. 
Note that the occurrences of $p$ are built only from the occurrences. 
As shown in Picture~\ref{pic-thm-pat-33}, suppose that there are $i$ ``isolated'' points in the  permutation, namely $a_1, a_2, \ldots, a_i$ and $i+1$ blocks, namely $A_0, A_1, \ldots, A_i$. Every pair of ``isolated'' points gives one occurrence of the pattern $p$ and each block should be \pattern{scale=0.5}{1}{1/1}{0/1,1/0}-avoiding, whose g.f. is given by Theorem~\ref{thm-length-1}.
Therefore,  the distribution is now given right away by 
$$F(x,q) = \sum_{i=0}^{\infty}q^{\binom{i}{2}} x^i 
\left(\frac{F(x)}{1+xF(x)} \right)^{i+1}.$$
In particular, $p$-avoiding permutations are given by the sum of the coefficients of $q^0$ and $q^1$:
$$A(x)=\frac{F(x)}{1+xF(x)} + x \left(\frac{F(x)}{1+xF(x)} \right)^2
=\frac{(1+2xF(x))F(x)}{(1+xF(x))^2}.$$
This completes the proof.
\end{proof}

\subsection{Distribution of the pattern Nr.\ 34} 
Our next theorem establishes the avoidance and the distribution of the pattern  
Nr.\ 34 = $\pattern{scale=0.6}{2}{1/1,2/2}{0/1,1/2,0/0,2/2,1/0,1/1,2/1}$.

\begin{thm}\label{thm-pat-34}
	Let $p=\pattern{scale=0.6}{2}{1/1,2/2}{0/1,1/2,0/0,2/2,1/0,1/1,2/1}$, 
	$F(x,q)=\sum_{n\geq 0}x^n\sum_{\pi\in S_n}q^{p(\pi)}$, and $A(x)$ be the g.f. for $S(p)$. Then, 
	\begin{align*}
	A(x)= \frac{F(x)}{1+x^2F(x)};
	\ \ \ \ \ 
	F(x,q)=\frac{F(x)}{1+(1-q)x^2F(x)}.
	\end{align*}
\end{thm}

\begin{figure}[!ht]
\begin{center}
	\begin{tikzpicture}[scale=0.8, baseline=(current bounding box.center)]
	\foreach \x/\y in {0/0,0/1,1/0,1/1,1/2,2/1,2/2}		    
	\fill[gray!20] (\x,\y) rectangle +(1,1);
	\draw (0.01,0.01) grid (2+0.99,2+0.99);
	\filldraw (1,1) circle (3pt) node[above left] {$a$};
	\filldraw (2,2) circle (3pt) node[above left] {$b$};
	\node  at (0.5,2.5) {$A$};
	\node  at (2.5,0.5) {$B$};
	\end{tikzpicture}
		\caption{Related to the proof of Theorem~\ref{thm-pat-34}}\label{pic-thm-pat-34}
\end{center}
\end{figure}
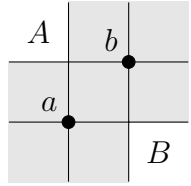

\begin{proof}
We claim that 
\begin{equation}\label{av-pattern-34} 
A(x) + x^2A(x)F(x)=F(x). 
\end{equation}
Indeed, each permutation $\pi$, which is counted by the right hand side in \eqref{av-pattern-34}, either avoids $p$ (which is counted by the $A(x)$ term in \eqref{av-pattern-34}), or  contains at least one occurrence of $p$. Among all such occurrences, pick an occurrence $ab$ with the \emph{leftmost} possible $a$ as shown in Fig.~\ref{pic-thm-pat-34}. Referring to this figure, we note that the permutation $A$ must be $p$-avoiding, or else, $a$ is not the leftmost possible. 
Further, $B$ can be any permutation giving the factor of $F(x)$ in $x^2A(x)F(x)$ in \eqref{av-pattern-34}. Finally, $a$ and $b$ contribute the factor of $x^2$.  Thus,  we complete the proof of   \eqref{av-pattern-34} leading to the formula for $A(x)$. 

For the distribution, we have the following functional equation:
\begin{equation}\label{dis-pattern-34} 
A(x) +x^2 q A(x)F(x,q)=F(x,q). 
\end{equation}
The proof of \eqref{dis-pattern-34} is essentially the same as that in the avoidance case. 
The block $B$ contributes $F(x,q)$, since all occurrences of $p$ in $B$ are preserved in the whole permutation.
Together with the factor $x^2q$ which corresponds to the elements $a$ and $b$, all the permutations containing occurrences of $p$ are counted by $x^2 q A(x)F(x,q)$. This completes the proof of 
\eqref{dis-pattern-34} and hence we get the formula for $A(x)$.
By substituting the formula for $A(x)$ found above into \eqref{dis-pattern-34}  we get the desired formula for $F(x,q)$.
\end{proof}

\subsection{Distribution of the pattern Nr.\ 55} 
Our next theorem establishes the avoidance and the distribution of the pattern  
Nr.\ 55 = $\pattern{scale = 0.6}{2}{1/1,2/2}{0/1,1/2,0/0,2/0,1/1,2/1}$.

\begin{thm}\label{thm-pat-55}
	Let $p=\pattern{scale = 0.6}{2}{1/1,2/2}{0/1,1/2,0/0,2/0,1/1,2/1}$, 
	$F(x,q)=\sum_{n\geq 0}x^n\sum_{\pi\in S_n}q^{p(\pi)}$, and $A(x)$ be the g.f. for $S(p)$. Then, 
	\begin{align*}
	A(x)= \frac{F(x)}{1+x(F(x)-1)};
	\ \ \ \ \ 
	F(x,q)= \frac{F(x)}{1+(1-q)x(F(x)-1)}.
	\end{align*}
\end{thm}

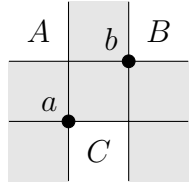
\begin{figure}[!ht]
\begin{center}
	\begin{tikzpicture}[scale=0.8, baseline=(current bounding box.center)]
	\foreach \x/\y in {0/0,0/1,1/1,1/2,2/0,2/1}		    
	\fill[gray!20] (\x,\y) rectangle +(1,1);
	\draw (0.01,0.01) grid (2+0.99,2+0.99);
	\filldraw (1,1) circle (3pt) node[above left] {$a$};
	\filldraw (2,2) circle (3pt) node[above left] {$b$};
	\node  at (0.5,2.5) {$A$};
	\node  at (2.5,2.5) {$B$};
	\node  at (1.5,0.5) {$C$};
	\end{tikzpicture}
		\caption{Related to the proof of Theorem~\ref{thm-pat-55}}\label{pic-thm-pat-55}
\end{center}
\end{figure}

\begin{proof}
	We claim that 
	\begin{equation}\label{av-pattern-55} 
	A(x) +x (F(x)-1)A(x)=F(x). 
	\end{equation}
	Indeed, each permutation $\pi$, which is counted by the right hand side in \eqref{av-pattern-55},  either avoids $p$ (which is counted by the $A(x)$ term in \eqref{av-pattern-55}), or  contains at least one occurrence of $p$. Among all such occurrences, pick an occurrence $ab$ where $a$ is  the \emph{highest} possible  such that $C$ is $p$-avoiding  as shown in Fig.~\ref{pic-thm-pat-55}. Hence, the block $C$ contributes the factor of $A(x)$.
	Note that, in general, $C$ can contain occurrences of $p$, but we can choose $C$ which is $p$-avoiding  of largest possible size.
	Because of $b$, no occurrence of $p$ can start in $C$ and end at $B$.
  	Referring to this figure, the blocks $A$ and $B$ together with $b$ can be any non-empty permutation and thus contribute the factor of $F(x)-1$.  
	Finally, $a$ contributes the factor of $x$.  Thus,  we complete the proof of  \eqref{av-pattern-55} and  hence give the formula for $A(x)$.

	We proceed to consider the distribution. 
	Let $B(x,q)$ be the g.f. for the distribution of $p$ on permutations where the element 1 is immediately followed by the element 2 (note that 12 is an occurrence of $p$). 
	We shall obtain an expression for $B(x,q)$ by seeing what happens if we remove 1 from each permutation discussed here.
	Generally, removing the element 1 (in fact any element) is not on the safe side in the sense that we could introduce one more occurrence of $p$.  However, in our case the element 2 is next to the element 1, so they have the same properties with respect to the other elements, and removing 1 is safe.
	If 2  is not immediately followed by 3, we lose one recurrence of $p$ recorded by $q$. Hence, we get the g.f. for the distribution on these permutations are counted by $qx \big(F(x,q)-1-B(x,q) \big)$, where $x$ is given by the element 1.
	If 2 is  immediately followed by 3, then we lose one occurrence of $p$, namely 12, but we also gain one occurrence of $p$ given by 23 in the original permutation.
	Hence, the permutations in this case are counted by $xB(x,q)$.
	Therefore, we have that 
	\begin{align*}
		B(x,q) = qx \big(F(x,q)-1-B(x,q) \big) +xB(x,q),
	\end{align*}
	and thus 
	\begin{align*}
		B(x,q) = \frac{qx\big(F(x,q)-1 \big)}{1-x+qx}.
	\end{align*}

	Finally,  we have the following functional equation for $F(x,q)$:
	\begin{equation}\label{dis-pattern-55} 
		F(x,q) = A(x) + B(x,q)A(x),
	\end{equation}
	since the structure of $A$ and $B$ together with $a$ and $B$ is the same as that of permutations appearing in $B(x,q)$.
	Substituting the formulas of $A(x)$  and $B(x,q)$ found above into \eqref{dis-pattern-55}   gives the desired formula for $F(x,q)$.
\end{proof}

\subsection{Distribution of the pattern Nr.\ 56} 
Our next theorem establishes the avoidance and  distribution of the pattern  
Nr.\ 56 = $\pattern{scale = 0.6}{2}{1/1,2/2}{0/1,1/2,0/0,2/2,1/1,2/1}$, which are the same as those of Nr.\ 55.

\begin{thm}\label{thm-pat-56}
	Let $p=\pattern{scale = 0.6}{2}{1/1,2/2}{0/1,1/2,0/0,2/2,1/1,2/1}$, 
	$F(x,q)=\sum_{n\geq 0}x^n\sum_{\pi\in S_n}q^{p(\pi)}$, and $A(x)$ be the g.f. for $S(p)$. Then, 
	\begin{align*}
	A(x) = \frac{F(x)}{1-x+xF(x)};
	\ \ \ \ \ 
	F(x,q) = \frac{F(x)}{1+(1-q)x(F(x)-1)}.
	\end{align*}
\end{thm}

\begin{figure}[!ht]
\begin{center}
	\begin{tikzpicture}[scale=0.8, baseline=(current bounding box.center)]
	\foreach \x/\y in {0/1,1/2,0/0,2/2,1/1,2/1}
	\fill[gray!20] (\x,\y) rectangle +(1,1);
	\draw (0.01,0.01) grid (2+0.99,2+0.99);
	\filldraw (1,1) circle (3pt) node[above left] {$a$};
	\filldraw (2,2) circle (3pt) node[above left] {$b$};
	\node  at (0.5,2.5) {$A$};
	\node  at (1.5,0.5) {$B$};
	\node  at (2.5,0.5) {$C$};
	\end{tikzpicture}
		\caption{Related to the proof of Theorem~\ref{thm-pat-56}}\label{pic-thm-pat-56}
\end{center}
\end{figure}
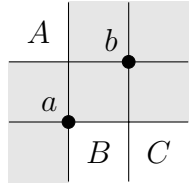

\begin{proof}
	We claim that 
\begin{equation}\label{av-pattern-56} 
A(x) +x (F(x)-1)A(x)=F(x). 
\end{equation}
Indeed, each permutation $\pi$, which is counted by the right hand side in \eqref{av-pattern-56}, either avoids $p$ (which is counted by the $A(x)$ term in \eqref{av-pattern-56}), or contains at least one occurrence of $p$.
Among all such occurrences, pick an occurrence $ab$ with the \emph{leftmost} possible $a$ as shown in Fig.~\ref{pic-thm-pat-56}. Referring to this figure,  the permutation $A$ must be $p$-avoiding, or else, $a$ is not the leftmost possible. 
Further, the blocks $B$ and $C$ together with $b$ can be any non-empty permutation and thus contribute the factor of $F(x)-1$.  
Finally, $a$ contributes the factor of $x$.  Thus,  we complete the proof of  \eqref{av-pattern-56} and  hence give the formula for $A(x)$.

We proceed to consider the distribution. 
Let $B(x,q)$ be the g.f. for the distribution of $p$ on permutations where the largest element is immediately before the second largest element (note that they form an occurrence of $p$). 
We next consider what happens if we remove the largest element from each permutation in $B(x,q)$.
If the third largest element is not  immediately before the second largest element, we lose one recurrence of $p$ recorded by $q$. Hence, we get the g.f. for the distribution on such permutations are counted by $qx \big(F(x,q)-1-B(x,q) \big)$, where $x$ is given by  the largest element.
If the third largest element is immediately before the second largest element, then we gain one occurrence of $p$ from these two elements, although one recurrence formed by the first two largest elements are lost.
Hence, the permutations discussed in this case are counted by $xB(x,q)$.
Therefore, we have that 
\begin{align*}
B(x,q) = qx \big(F(x,q)-1-B(x,q) \big) +xB(x,q),
\end{align*}
and thus, 
\begin{align*}
B(x,q) = \frac{qx\big(F(x,q)-1 \big)}{1-x+qx}.
\end{align*}

\begin{figure}[!htb]
	\centering
\begin{minipage}{0.25\textwidth}
				\centering
	\begin{tikzpicture}[scale=0.8, baseline=(current bounding box.center)]
	\foreach \x/\y in {1/1,2/1}
	\foreach \x/\y in {1/0,2/0}		    \draw (\x,\y) rectangle +(0.9,0.9);
	\filldraw (1.95,1.1) circle (2pt) node[above] {$b$};
	\node  at (1.45,0.5) {$B$};
	\node  at (2.45,0.5) {$C$};
	\end{tikzpicture}
\end{minipage}
	\begin{minipage}{0.05\textwidth}
		\centering
		$\longrightarrow$	
	\end{minipage}
\begin{minipage}{0.25\textwidth}
			\centering
	\begin{tikzpicture}[scale=0.8, baseline=(current bounding box.center)]
	\foreach \x/\y in {1/1,2/1}
	\foreach \x/\y in {1/0,2/0}		    \draw (\x,\y) rectangle +(0.9,0.9);
	\filldraw (0.95,1.05) circle (2pt) node[above] {$a$};
	\filldraw (1.95,1.25) circle (2pt) node[above] {$b$};
	\node  at (1.45,0.5) {$B$};
	\node  at (2.45,0.5) {$C$};
	\end{tikzpicture}
\end{minipage}
\begin{minipage}{.1\textwidth}
	\centering
	$\longrightarrow$	
\end{minipage}
\begin{minipage}{0.25\textwidth}
				\centering
	\begin{tikzpicture}[scale=0.8, baseline=(current bounding box.center)]
\foreach \x/\y in {1/1,2/1}
\foreach \x/\y in {1/0,2/0,0/1.5}		    \draw (\x,\y) rectangle +(0.9,0.9);
\filldraw (0.95,1.05) circle (2pt) node[above] {$a$};
\filldraw (1.95,1.25) circle (2pt) node[above] {$b$};
\node  at (0.45,2) {$A$};
\node  at (1.45,0.5) {$B$};
\node  at (2.45,0.5) {$C$};
\end{tikzpicture}
\end{minipage}
	\caption{Related to the proof of Theorem~\ref{thm-pat-56}}\label{pic-thm-pat-56-2}
\end{figure}
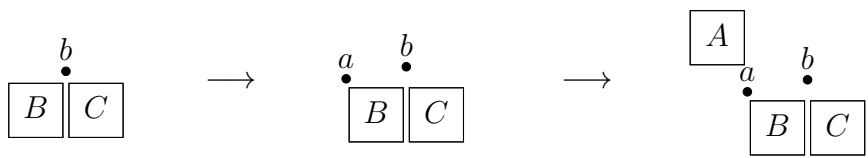

Finally, 
We  have the following functional equation for $F(x,q)$ by using similar steps in deriving $B(x,q)$:
\begin{equation}\label{dis-pattern-56} 
F(x,q) = A(x) + x B(x,q)A(x) + xq \big( F(x,q)-1-B(x,q) \big),
\end{equation}
Indeed, all the  permutations with at least one occurrences of $p$ can be generated as shown in Fig.~\ref{pic-thm-pat-56-2}. We shall consider what happens after inserting $a$ in front of $B$. If b is preceded immediately by $b-1$, then we lose one occurrence of $p$, which is $(b-1)b$, but we gain $ab$, and hence this corresponds to the second item on the right hand side,  where $x$ is given by $a$ and $A(x)$ is given by the block $A$.
If $b$ is not preceded immediately by $b-1$, we doe not lose anything but gain $ab$ as one occurrence of $p$, and hence this corresponds to the third item on the right hand side.
By substituting the formulas of $A(x)$  and $B(x,q)$ found above into \eqref{dis-pattern-56}  we  get  the desired formula for $F(x,q)$.
\end{proof}

\subsection{Distribution of the pattern Nr.\ 63} 
Our next theorem establishes the avoidance and  distribution of the pattern  
Nr.\ 63 = $\pattern{scale = 0.6}{2}{1/1,2/2}{0/1,1/2,0/0,2/1,2/0}$.

\begin{thm}\label{thm-pat-63}
	Let $p=\pattern{scale = 0.6}{2}{1/1,2/2}{0/1,1/2,0/0,2/1,2/0}$, 
	$F(x,q)=\sum_{n\geq 0}x^n\sum_{\pi\in S_n}q^{p(\pi)}$, and $A(x)$ be the g.f. for $S(p)$. Then, 
	\begin{align*}
	A(x) = \frac{2F(x)-1}{F(x)};
	\ \ \ \ \ 
	F(x,q) = \frac{(2-q)F(x)+q-1}{(1-q)F(x)+q}.
	\end{align*}
\end{thm}

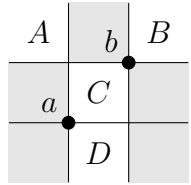
\begin{figure}[!ht]
\begin{center}
	\begin{tikzpicture}[scale=0.8, baseline=(current bounding box.center)]
	\foreach \x/\y in {0/1,1/2,0/0,2/1,2/0}
	\fill[gray!20] (\x,\y) rectangle +(1,1);
	\draw (0.01,0.01) grid (2+0.99,2+0.99);
	\filldraw (1,1) circle (3pt) node[above left] {$a$};
	\filldraw (2,2) circle (3pt) node[above left] {$b$};
	\node  at (0.5,2.5) {$A$};
	\node  at (2.5,2.5) {$B$};
	\node  at (1.5,1.5) {$C$};
	\node  at (1.5,0.5) {$D$};
	\end{tikzpicture}
		\caption{Related to the proof of Theorem~\ref{thm-pat-63}}\label{pic-thm-pat-63}
\end{center}
\end{figure}

\begin{proof}
We claim that 
\begin{equation}\label{av-pattern-63} 
A(x) + \big(A(x)-1\big) \big(F(x)-1\big)=F(x). 
\end{equation}
Indeed, each permutation $\pi$, which is counted by the right hand side in \eqref{av-pattern-63},  either avoids $p$ (which is counted by the $A(x)$ term in \eqref{av-pattern-63}), or  contains at least one occurrence of $p$. Among all such occurrences, pick the occurrence $ab$ with the \emph{highest} possible $b$ as shown in Fig.~\ref{pic-thm-pat-63}. Referring to this figure, we note that the permutation formed by $A$ and $B$ together with $b$ must be nonempty and $p$-avoiding. This explains the term $A(x)-1$ in \eqref{av-pattern-63}. Further, since $ab$ is an occurrence of $p$, in the part below $b$ which is formed by $C$ and $D$ together with $a$, we must have a non-empty permutation, which can be any, and such permutations are counted by $F(x)-1$. This completes the proof of   \eqref{av-pattern-63} and gives the formula for $A(x)$. 

For the distribution, we have the following functional equation:
\begin{equation}\label{dis-pattern-63} 
A(x) + q \big(A(x)-1\big) \big(F(x,q)-1\big)=F(x,q). 
\end{equation}
The proof of \eqref{dis-pattern-63} is essentially the same as that in the avoidance case. 
The term $q$ is given by $ab$ and the non-empty part below $b$ gives $F(x,q)-1$ since no occurrence of $p$ can start there and end at $x>b$ (or else $bx$ would be an occurrence of $p$ contradicting $b$ being the highest possible). Solving for $F(x,q)$ and substituting $A(x)$, we obtain the desired formula for $F(x,q)$. \end{proof}

\subsection{Distribution of the pattern Nr.\ 64}
Our next theorem establishes the avoidance and the distribution of the pattern  
Nr.\ 64 = $\pattern{scale = 0.6}{2}{1/1,2/2}{0/1,1/2,2/0,0/2,1/1}$. Note that the occurrences of the pattern Nr.\ 64 are related to the well known statistic  ``the number of components'', which in turn is related to the notion of an ``irreducible permutation'', or an ``indecomposable permutation''. More precisely, the number of occurrences of the pattern Nr.\ 64 equals the number of  components minus 1, which is the number of places in which a permutation can be cut so that every element to the left of a cut is less than any element to the right of the cut. This distribution is known (see the sequence A059438 in \cite{OEIS}). However,  we rederive it in the next theorem in a different form, which will allow us  to establish, in a bijective way, the equidistribution with the pattern Nr.\ 63, and the pattern Nr.\ 65 considered in Theorem~\ref{thm-pat-65}.

\begin{thm}\label{thm-pat-64}
	Let $p=\pattern{scale = 0.6}{2}{1/1,2/2}{0/1,1/2,2/0,0/2,1/1}$, 
	$F(x,q)=\sum_{n\geq 0}x^n\sum_{\pi\in S_n}q^{p(\pi)}$, and $A(x)$ be the g.f. for $S(p)$. Then, 
	\begin{align*}
	A(x)=\frac{2F(x)-1}{F(x)};
	\ \ \ \ \ 
	F(x,q)= \frac{(2-q)F(x)+q-1}{(1-q)F(x)+q}.
	\end{align*}
\end{thm}

\begin{figure}[!ht]
\begin{center}
	\begin{tikzpicture}[scale=0.8, baseline=(current bounding box.center)]
	\foreach \x/\y in {0/1,1/2,2/0,0/2,1/1}
	\fill[gray!20] (\x,\y) rectangle +(1,1);
	\draw (0.01,0.01) grid (2+0.99,2+0.99);
	\filldraw (1,1) circle (3pt) node[above left] {$a$};
	\filldraw (2,2) circle (3pt) node[above left] {$b$};
	\node  at (2.5,2.5) {$C$};
	\node  at (2.5,1.5) {$D$};
	\node  at (0.5,0.5) {$A$};
	\node  at (1.5,0.5) {$B$};
	\end{tikzpicture}
		\caption{Related to the proof of Theorem~\ref{thm-pat-64}}\label{av-pic-pat-64}
\end{center}
\end{figure}
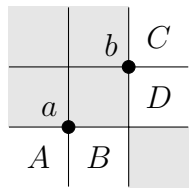

\begin{proof}

We claim that 
\begin{equation}\label{av-pattern-64} 
A(x) + \big(A(x)-1\big) \big(F(x)-1\big)=F(x). 
\end{equation}
Indeed, each permutation $\pi$, which is counted by the right hand side in \eqref{av-pattern-64}, either avoids $p$ (which is counted by the $A(x)$ term in \eqref{av-pattern-64}), or  contains at least one occurrence of $p$. Among all such occurrences, pick the occurrence $ab$ with the \emph{leftmost} possible $a$, which determines uniquely $b$, as shown in Fig.~\ref{av-pic-pat-64}. Referring to this figure, we note that the permutation formed by $a$, $A$ and $B$ must be non-empty $p$-avoidable. This explains the term $A(x)-1$ in \eqref{av-pattern-64}. On the other hand, the non-empty permutation formed by $b$, $C$ and $D$ can be any, which explains the term $F(x)-1$ completing our proof of   \eqref{av-pattern-64} and giving the formula for $A(x)$. 

For the distribution, we have the following functional equation:
\begin{equation}\label{dis-pattern-64} 
A(x) + q \big(A(x)-1\big) \big(F(x,q)-1\big)=F(x,q). 
\end{equation}
The proof of \eqref{dis-pattern-64} is essentially the same as that in the avoidance case. 
The term $q$ is given by $ab$ and the non-empty part to the right of $B$ contributes $F(x,q)-1$ because no occurrence of $p$ can start to the left of $b$ and end to the right of $b$. Solving for $F(x,q)$ and substituting $A(x)$, we obtain the desired formula for $F(x,q)$. 
\end{proof}

\begin{remark}\label{bijective-remark-1} {\em Comparing the structures in Fig.~\ref{pic-thm-pat-63} and~\ref{av-pic-pat-64}, we can explain the equidistribution of the patterns $p_1=$ Nr.\ $63$ and $p_2=$ Nr.\ $64$ bijectively, where using \eqref{simple-observ} given below and the discussion around it, we can map $p_2$-avoiding permutations to $p_1$-avoiding permutations, say, lexicographically, thus having the basis of the recursion. Skipping the details, the idea of the bijective map is to start with a permutation of the form in Fig.~\ref{av-pic-pat-64} with $k$ occurrences of $p_2$, then 
 \begin{itemize} 
 \item[a.] map recursively the permutation formed by $b,C,D$ with $(k-1)$ occurrences of $p_2$ to the permutation formed by $a, C, D$ in Fig.~\ref{pic-thm-pat-63} with $(k-1)$ occurrences of $p_1$;
 \item[b.] map the $p_2$-avoiding permutation formed by $a$, $A$ and $B$ in Fig.~\ref{av-pic-pat-64} to a $p_1$-avoiding permutation formed by $b$, $A$ and $B$  in Fig.~\ref{pic-thm-pat-63};
 \item[c.] combine the permutations obtained in a.\ and b.\ to obtain the structure in Fig.~\ref{pic-thm-pat-63} giving the desired permutation with $k$ occurrences of $p_2$.
 \end{itemize}
 }
 \end{remark}

\subsection{Distribution of the pattern Nr.\ 65} 
Our next theorem establishes the avoidance and the distribution of the pattern  
Nr.\ 65 = $\pattern{scale = 0.6}{2}{1/1,2/2}{0/1,1/0,0/0,1/1,2/2}$

\begin{thm}\label{thm-pat-65}
	Let $p=\pattern{scale = 0.6}{2}{1/1,2/2}{0/1,1/0,0/0,1/1,2/2}$, 
	$F(x,q)=\sum_{n\geq 0}x^n\sum_{\pi\in S_n}q^{p(\pi)}$, and $A(x)$ be the g.f. for $S(p)$. Then, 
	\begin{align*}
	A(x)=\frac{2F(x)-1}{F(x)};
	\ \ \ \ \ 
	F(x,q)= \frac{(2-q)F(x)+q-1}{(1-q)F(x)+q}.
	\end{align*}
\end{thm}

\begin{figure}[!ht]
\begin{center}
	\begin{tikzpicture}[scale=0.8, baseline=(current bounding box.center)]
	\foreach \x/\y in {0/1,1/0,0/0,1/1,2/2}
	\fill[gray!20] (\x,\y) rectangle +(1,1);
	\draw (0.01,0.01) grid (2+0.99,2+0.99);
	\filldraw (1,1) circle (3pt) node[above left] {$a$};
	\filldraw (2,2) circle (3pt) node[above right] {$b$};
	\node  at (0.5,2.5) {$A$};
	\node  at (1.5,2.5) {$B$};
	\node  at (2.5,1.5) {$C$};
	\node  at (2.5,0.5) {$D$};
	\end{tikzpicture}
		\caption{Related to the proof of Theorem~\ref{thm-pat-65}}\label{pic-thm-pat-65-a}
\end{center}
\end{figure}

\begin{proof}
We claim that 
\begin{equation}\label{av-pattern-65} 
A(x) + \big(A(x)-1\big) \big(F(x)-1\big)=F(x). 
\end{equation}
Indeed, each permutation $\pi$, which is counted by the right hand side in \eqref{av-pattern-65}, either avoids $p$ (which is counted by the $A(x)$ term in \eqref{av-pattern-65}), or  contains at least one occurrence of $p$. Among all such occurrences, pick an occurrence $ab$ with the \emph{highest} possible $b$ as shown in Fig.~\ref{pic-thm-pat-65-a}. 

In order for $b$ to be the highest possible, the non-empty permutation formed by $a$, $A$ and $B$ must be $p$-avoiding, which explains the term $A(x)-1$ in  \eqref{av-pattern-65}. On the other hand, the non-empty permutation formed by $a$, $C$ and $D$ can be any, which explains  the term $F(x)-1$ in  \eqref{av-pattern-65} (the presence of $b$ leads to no problem since we have used $a$ twice). We are done with proving \eqref{av-pattern-65}, because we can construct  $\pi$, in a bijective way, from a $p$-avoiding non-empty permutation $\pi'$ and from another non-empty permutation $\pi''$ as is sketched in Fig.~\ref{pic-thm-pat-65}. Indeed,  the minimal element $a'$ in $\pi'$ and the leftmost element $a''$ in $\pi''$ will give the value and position of the element $a$ in $\pi$ as shown in Fig.~\ref{pic-thm-pat-65}; then $a'$ and $a''$ can be removed and the element $b$ can be inserted in the uniquely defined position; this procedure is reversible. 

\begin{figure}[!htb]
	\centering
	\begin{tikzpicture}[scale=1, baseline=(current bounding box.center)]
	\foreach \x/\y in {2/0,2/1}		    \draw (\x,\y) rectangle +(0.85,0.85);
	\foreach \x/\y in {0/2,1/2}		    \draw (\x,\y) rectangle +(0.85,0.85);
	\filldraw (0.9,0.9) circle (1.5pt) node[below left] {$a$};
	\filldraw (1.9,0.9) circle (1.5pt) node[below left] {$a''$};
	\filldraw (0.9,1.9) circle (1.5pt) node[below left] {$a'$};
	\filldraw (1.9,1.9) circle (1.5pt) node[above right] {$b$};
	\draw[<-,line width=0.7pt] (1.1,0.9) -- (1.7,0.9);		
	\draw[<-,line width=0.7pt] (0.9,1.1) -- (0.9,1.7);
	\node  at (0.45,2.5) {$A$};\node  at (1.45,2.5) {$B$};
	\node  at (2.45,1.5) {$C$};\node  at (2.45,0.5) {$D$};
	\end{tikzpicture}
	\caption{Related to the proof of Theorem~\ref{thm-pat-65}}\label{pic-thm-pat-65}
\end{figure}

For the distribution, we have the following functional equation:
\begin{equation}\label{dis-pattern-65} 
A(x) + q \big(A(x)-1\big) \big(F(x,q)-1\big)=F(x,q). 
\end{equation}
The proof of \eqref{dis-pattern-16} is essentially the same as that in the avoidance case. 
The term $q$ is given by $ab$; the term $A(x)-1$ is given by  the non-empty permutation formed by $a'$, $A$ and $B$, and the term $F(x,q)-1$ is given by  the non-empty permutation formed by $a''$, $C$ and $D$, since no occurrence of $p$ can begin at $b$, or above it. Solving for $F(x,q)$ and substituting $A(x)$, we obtain the desired formula for $F(x,q)$. 
 \end{proof}
 
 \begin{remark}\label{bijective-remark-2} {\em Comparing the structures in Fig.~\ref{av-pic-pat-64} and~\ref{pic-thm-pat-65-a}, we can explain the equidistribution of the patterns $p_1=$ Nr.\ $64$ and $p_2=$ Nr.\ $65$ bijectively, where using \eqref{simple-observ} given below and the discussion around it, we can map $p_1$-avoiding permutations to $p_2$-avoiding permutations, say, lexicographically thus having the basis of the recursion. Skipping the details, the idea of the bijective map is to start with a permutation of the form in Fig.~\ref{av-pic-pat-64} with $k$ occurrences of $p_1$, then 
 \begin{itemize} 
 \item[a.] map recursively the permutation formed by $b,C,D$ with $(k-1)$ occurrences of $p_1$ to the permutation formed by $a, C, D$ in Fig.~\ref{pic-thm-pat-65-a} with $(k-1)$ occurrences of $p_2$ (this map will give us the value of $a$);
 \item[b.] map the $p_1$-avoiding permutation formed by $a$, $A$ and $B$ in Fig.~\ref{av-pic-pat-64} to a $p_2$-avoiding permutation formed by the same letters in Fig.~\ref{pic-thm-pat-65-a} (this map will give us the position of $a$);
 \item[c.] glue the elements $a$ obtained in a.\ and b.\ and insert $b$ in the unique position in east-south of $B$ and west-north of $C$ in Fig.~\ref{pic-thm-pat-65-a} to obtain the desired permutation with $k$ occurrences of $p_2$.
 \end{itemize}
 }
 \end{remark}

\section{Distributions via recurrence relations}\label{sec-4}

In this section, we consider six patterns for which our generating functions approach does not work. Instead, we derive recurrence relations for the distribution of  respective patterns. The patterns are:
\begin{center}
\begin{tabular}{rrr}
	Nr.\ 8  = $\pattern{scale = 0.6}{2}{1/1,2/2}{0/0,0/1,1/0,1/1}$; &
	Nr.\ 9 = $\pattern{scale = 0.6}{2}{1/1,2/2}{0/1,1/1,1/2,2/1}$; &
	Nr.\ 14 = $\pattern{scale = 0.6}{2}{1/1,2/2}{0/1,1/1,1/2,1/0,1/2,2/1}$; \\[5pt]
	Nr.\ 15 = $\pattern{scale = 0.6}{2}{1/1,2/2}{0/1,0/2,1/0,1/1,1/2}$;  &
	Nr.\ 36 = $\pattern{scale=0.6}{2}{1/1,2/2}{0/1,1/2,0/0,1/0,1/1,2/1}$;  &
	Nr.\ 45 = $\pattern{scale=0.6}{2}{1/1,2/2}{0/1,1/2,1/0,1/1,2/1,0/2}$.  
\end{tabular}
\end{center}

\noindent
In this section, we denote by $T_{n,k}$  the number of $n$-permutations with $k$ occurrences of the pattern in question. Also, let $T_n(x)=\sum_{k= 0}^{n-1}T_{n,k}x^k$.

\subsection{Distributions of the patterns Nr.\ 8 and Nr.\ 9}
In this section we show that the distributions for the patterns Nr.\ 8 = $\pattern{scale = 0.6}{2}{1/1,2/2}{0/0,0/1,1/0,1/1}$ and Nr.\ 9 = $\pattern{scale = 0.6}{2}{1/1,2/2}{0/1,1/1,1/2,2/1}$ are given by the {\em unsigned Stirling numbers of the first kind} (see the sequence A132393 in \cite{OEIS}). In the proof of the next theorem we think of generating all $n$-permutations from all $(n-1)$-permutations first by inserting, in all possible places,  the largest element, and then by inserting the smallest element. 

\begin{thm}\label{dis-patterns-8-9}
	Both patterns $p_1=\pattern{scale = 0.6}{2}{1/1,2/2}{0/0,0/1,1/0,1/1}$ and  $p_2=\pattern{scale = 0.6}{2}{1/1,2/2}{0/1,1/1,1/2,2/1}$ satisfy 
	\begin{align}
	T_{n,k} = & \, T_{n-1,k-1}+ (n-1)T_{n-1,k}  \label{rec-rel-8-9}
	\end{align}
	with  the initial conditions $T_{n,0}=(n-1)!$ for $n\geq 1$ and $T_{0,0}=1$,  which shows that $T_{n,k}=C(n,k+1)$, the unsigned Stirling number of the first kind.	The row generating function for $T_{n,k}$ is given by
	\begin{align}
	\sum_{k=0}^{n-1}T_{n,k}x^k=\prod_{i=1}^{n-1}(x+i).\label{gf-rec-rel-8-9}
	\end{align}
\end{thm}

\begin{proof}
Let us consider $p_1$. We note that avoidance for $p_1$ is given in \cite[Prop. 18]{Hilmarsson2015Wilf}, but it is easy to see directly that $T_{n,0}=(n-1)T_{n-1,0}$ for $n\geq 2$ and $T_{1,0}=1$, so that $T_{n,0}=(n-1)!$ for $n\geq 1$. Indeed, inserting the largest element $n$ in an $(n-1)$-permutation never decreases the number of occurrences of $p_1$. Furthermore, inserting $n$ in position 2 introduces an extra occurrence of $p_1$, and inserting $n$ in any other position preserves the number of occurrences of $p_1$.  These observations do not only prove the avoidance case (the initial condition), but also (\ref{rec-rel-8-9}).

For the pattern $p_2$, the arguments are very similar, except that we insert the smallest element $x$ instead of the largest element, and the only place when inserting $x$ increases the number of occurrences of $p_2$ by 1 is immediately before the minimal element 1.  The base case (avoidance) is given by  \cite[Prop. 19]{Hilmarsson2015Wilf}.

Finally, (\ref{gf-rec-rel-8-9}) follows from the well known properties of the unsigned Stirling numbers of the first kind \cite[A132393]{OEIS}; also, see \cite{Stanley}. \end{proof}
%
%
%
%

\subsection{Distributions of the patterns Nr.\ 14 and Nr.\ 15}\label{subsec-14-15}
An occurrence of the pattern Nr.\ 14 = $\pattern{scale = 0.6}{2}{1/1,2/2}{0/1,1/1,1/2,1/0,1/2,2/1}$ is known  as a {\em small ascent}, and its reverse as a {\em small descent}. The distribution of this pattern is given by the sequence A123513 in \cite{OEIS} with a reference to  $E_{n,1}(x)$ in Table 1 on page 291 in \cite{LR}. The next theorem derives a recurrence relation for the pattern and shows that the same recurrence relation works for the pattern Nr.\ 15 = $\pattern{scale = 0.6}{2}{1/1,2/2}{0/1,0/2,1/0,1/1,1/2}$ thus establishing equidistribution of these patterns.  
 
%
%

\begin{thm}\label{dis-patterns-14-15}
	Both patterns $p_1=\pattern{scale = 0.6}{2}{1/1,2/2}{0/1,1/1,1/2,1/0,1/2,2/1}$ and  $p_2=\pattern{scale = 0.6}{2}{1/1,2/2}{0/1,0/2,1/0,1/1,1/2}$ satisfy the recurrence relation  
	\begin{align}
	T_{n,k} = & \, T_{n-1,k-1}+ (k+1)T_{n-1,k+1} + (n-k-1)T_{n-1,k} \label{rec-rel-14-15}
	\end{align}
	with  the initial conditions	$T_{1, 0} = 1, T_{2, 0} = 1, T_{2, 1} = 1$.
	Equivalently,
	\begin{align}
	T_n(x) =(x+n-1) T_{n-1}(x)+(1-x) T'_{n-1}(x)\label{gf-rec-rel-14-15}
	\end{align}
	with the initial conditions $T_1(x)=1$ and $T_2(x)=1+x$. 
\end{thm}

\begin{proof}
The initial conditions are easy to see for both $p_1$ and $p_2$. Next, we explain (\ref{rec-rel-14-15})	 for $p_1$ and $p_2$ thinking of generating all $n$-permutations counted by $T_{n,k}$ by inserting the new smallest element $x$ in an $(n-1)$-permutation.

For $p_1$, $T_{n-1,k-1}$ is the number of possibilities to insert $x$ right in front of the element 1, which increases the number of occurrences of $p_1$ by 1. Further, $(k+1)T_{n-1,k+1}$ counts the possibilities to pick one of the $(k+1)$ occurrences of $p_1$, say $ab$ (which is formed by consecutive elements) and remove it by inserting $x$ between $a$ and $b$. Finally, $(n-k-1)T_{n-1,k}$ counts the number of possibilities to insert $x$ and not to change the number of occurrences of $p_1$, where in $(n-k-1)$, $n$ is the number of possibilities to insert $x$, $k$ is the number of possibilities to decrease the number of occurrences of $p_1$, and $1$ is the number of possibilities to increase the number of occurrences of $p_1$.

For $p_2$, explanation for all the terms in (\ref{rec-rel-14-15}) is the same except $T_{n-1,k-1}$, which corresponds to inserting $x$ in the leftmost position, which increases the number of occurrences of $p_2$ by 1. 
	
We now obtain (\ref{gf-rec-rel-14-15}) from (\ref{rec-rel-14-15})	 as follows:
	\begin{align*}
	T_n(x) = & \, \sum_{k=0}^{n-1} T_{n-1,k-1}  x^k + \sum_{k=0}^{n-1} (k+1)T_{n-1,k+1} x^k  +\sum_{k=0}^{n-1} (n-k-1)T_{n-1,k} x^k \\[6pt]
	= &  \, x T_{n-1}(x) +  T'_{n-1}(x) + (n-1) T_{n-1}(x) - x  T'_{n-1}(x)\\[6pt]
	= & \, (x+n-1) T_{n-1}(x)+(1-x) T'_{n-1}(x).
	\end{align*}
This completes the proof. \end{proof}

\subsection{Distribution of the pattern Nr.\ 36}
In this section, we find the recurrence relation for the distribution of the pattern Nr.\ 36 = $\pattern{scale=0.6}{2}{1/1,2/2}{0/1,1/2,0/0,1/0,1/1,2/1}$. Note that an occurrence of the pattern is a small ascent (see Section~\ref{subsec-14-15}) in which the left element is a {\em left-to-right minimum}, that is, an element having no smaller elements to the left of it.


\begin{thm}\label{dis-patterns-36}
	The pattern $p=\pattern{scale=0.6}{2}{1/1,2/2}{0/1,1/2,0/0,1/0,1/1,2/1}$ satisfies the recurrence relation  
	\begin{align}
		T_{n,k} = & \, (k+1)T_{n-1,k+1}+(n-k)T_{n-1,k}  -T_{n-2,k}+T_{n-2,k-1}\label{rec-rel-36}
\end{align}
	with  the initial conditions	$T_{1, 0} = 1, T_{2, 0} = 1, T_{2, 1} = 1$.
Equivalently,
	\begin{align}
		T_n(x) = n T_{n-1}(x)+(1-x) T'_{n-1}(x)+(x-1)T_{n-2}(x)\label{gf-rec-rel-36}
	\end{align}
with the initial conditions $T_1(x)=1$ and $T_2(x)=1+x$. 
\end{thm}

\begin{proof} To explain (\ref{rec-rel-36}), we think of generating all $n$-permutations from $(n-1)$-permutations by inserting the new largest element $n$. We have:
\begin{itemize}
\item $T_{n-2,k-1}$ corresponds to inserting $n$ right after the leftmost element $x$ when $x=n-1$. In this case, $(n-1)n$ is an extra occurrence of $p$, while before inserting, $n-1$ is not involved in an occurrence of $p$.
\item $(k+1)T_{n-1,k+1}$ counts the possibilities to pick one of the $(k+1)$ occurrences of $p$, say $ab$ (which is formed by consecutive elements) and remove it by inserting $n$ between $a$ and $b$.  
\item $(n-k)T_{n-1,k}$ counts the number of possibilities to insert $x$ and not to change the number of occurrences of $p$, where in $(n-k)$, $n$ is the number of possibilities to insert $x$, and $k$ is the number of possibilities to decrease the number of occurrences of $p$.
\item $-T_{n-2,k}$ corresponds to the fact that there is over-counting in the term $(n-k)T_{n-1,k}$. Indeed, if an $(n-1)$-permutation begins with $(n-1)$ then the position immediately to the right of $(n-1)$ is counted by $(n-k)$. However, inserting $n$ in this position actually increases the number of occurrences of $p$.
\end{itemize}
	
We now obtain (\ref{gf-rec-rel-36}) from (\ref{rec-rel-36})  as follows:
	\begin{align*}
	T_n(x) = & \, \sum_{k=0}^{n-1} (k+1)T_{n-1,k+1} x^k + \sum_{k=0}^{n-1} (n-k)T_{n-1,k} x^k \\[5pt]
	& \,   -\sum_{k=0}^{n-1} T_{n-2,k} x^k + \sum_{k=0}^{n-1} T_{n-2,k-1} x^k\\[6pt]
	= &  \, T'_{n-1}(x) +  n T_{n-1}(x)- x  T'_{n-1}(x)- T_{n-2}(x)+ x T_{n-2}(x)\\[6pt]
	= & \, n T_{n-1}(x)+(1-x) T'_{n-1}(x)+(x-1)T_{n-2}(x).
	\end{align*}
	This completes the proof.
\end{proof}

\subsection{Distribution of the pattern Nr.\ 45}
In this section, we find the recurrence relation for the distribution of the pattern Nr.\ 45 = $\pattern{scale=0.6}{2}{1/1,2/2}{0/1,1/2,1/0,1/1,2/1,0/2}$, which is the most difficult case in Section~\ref{sec-4}.  


\begin{thm}\label{dis-patterns-45}
	The pattern $p=\pattern{scale=0.6}{2}{1/1,2/2}{0/1,1/2,1/0,1/1,2/1,0/2}$ satisfies the recurrence relation 
	\begin{align}
	T_{n,k} & =   (k + 1) T_{n - 1, k + 1} + (n - k - 1) T_{n - 1, k}   + T_{n - 1, k - 1}  \notag \\[5pt]
	& \quad + (k + 1) T_{n - 2, k + 1} + (n - 2 k - 2) T_{n - 2, k} - (n - k - 1) T_{n - 2, k - 1}
	\label{main-for-45}
	\end{align}
with  the initial conditions	$T_{1, 0} = 1, T_{2, 0} = 1, T_{2, 1} = 1$.
Equivalently,
\begin{align}
T_n(x)  = & (x+n-1)  T_{n-1}(x) + (1-x)  T'_{n-1}(x) \notag\\[5pt]
& +(n-2) (1-x) T_{n-2}(x) + (1-x)^2 T'_{n-2}(x), \label{eq:T1}
\end{align}
with the initial conditions $T_1(x)=1$ and $T_2(x)=1+x$.
\end{thm}	

\begin{proof}
The initial conditions are easy to check. 

Let $B_{n,k}$ be the number of $n$-permutations beginning with the smallest element 1 and having $k$ occurrences of $p$. We claim that 
\begin{equation}
B_{n,k}=B_{n-1,k-1}+T_{n-1,k}-B_{n-1,k}.
\label{B-for-45}
\end{equation}
Indeed, it is not difficult to see that $B_{n-1,k-1}$ is the number of $n$-permutations that begin with  12, because $12$ is an occurrence of $p$. Further, $T_{n-1,k}-B_{n-1,k}$ counts those $n$-permutations counted by $B_{n,k}$ that do not begin with $12$, which is easy to see by removing 1 and decreasing by 1 any other element in each such permutation.
 
Next, we claim that  
\begin{equation}\label{T-for-45}
T_{n,k}=B_{n-1,k-1}+(k+1)T_{n-1,k+1}+(n-k)T_{n-1,k}-B_{n-1,k}.
\end{equation}
Indeed, we can think of creating all $n$-permutations from all $(n-1)$-permutations by inserting the new smallest element $x$ in all possible places. The terms in (\ref{T-for-45}) are then explained as follows. 

\begin{itemize}
\item $B_{n-1,k-1}$ describes the situation when an $(n-1)$-permutation begins with the smallest element 1, and $x$ is inserted at the very beginning (immediately to the left of $1$). 
\item $(k+1)T_{n-1,k+1}$ describes the situation when an occurrence $ab$ of $p$ is replaced by $axb$ thus eliminating one occurrence of $p$. 
\item $(n-k)T_{n-1,k}-B_{n-1,k}$ describes the situation when inserting $x$ does not change the number of occurrences of $p$. For any $(n-1)$-permutation with $k$ occurrences of $p$ there are $(n-k)$ places to insert $x$ except when the $(n-1)$-permutation begins with 1,  because inserting $x$ immediately before $1$ actually increases the number of occurrences of $p$. 
\end{itemize}
 
From (\ref{T-for-45}) and (\ref{B-for-45}) we have
\begin{align}
	T_{n,k} = B_{n,k}+(k+1)T_{n-1,k+1}+(n-k-1)T_{n-1,k}.
	\label{mixed-B-T-45}
\end{align}
Let $B_{n}(x)=\sum_{k=0}^{n-1}B_{n,k}x^k$. From (\ref{mixed-B-T-45}), we have 
\begin{align*}
	T_n(x)  = &  \, B_{n}(x) +   T'_{n-1}(x) + (n-1) T_{n-1}(x) - x  T'_{n-1}(x) \notag \\[5pt]
				= & \, B_{n}(x) + (n-1) T_{n-1}(x) + (1-x)  T'_{n-1}(x).
\end{align*}
Hence,
\begin{align}
	T_n(x) - (x-1)T_{n-1}(x)  = & \, B_{n}(x) + (1-x) B_{n-1}(x)\notag\\[5pt]
	& + (n-1) T_{n-1}(x) + (1-x)  T'_{n-1}(x) \notag\\[5pt]
	& + (n-2) (1-x) T_{n-2}(x) + (1-x)^2 T'_{n-2}(x). \label{eq:T11234}
\end{align}
On the other hand, from (\ref{B-for-45}), we have
\begin{align}
B_{n}(x)  = &  \, x B_{n-1}(x) +   T_{n-1}(x) - B_{n-1}(x) \notag \\[5pt]
				 = &  \,  T_{n-1}(x) + (x-1) B_{n-1}(x).\label{eq:1010101}
\end{align}	
By \eqref{eq:1010101}, we can replace  $B_{n}(x) + (1-x) B_{n-1}(x)$ in \eqref{eq:T11234} by $T_{n-1}(x)$, which completes the proof of \eqref{eq:T1}. The equation \eqref{main-for-45} is obtained from \eqref{eq:T1} by taking the coefficients of both sides.
\end{proof}

\section{An equidistribution result} 

For a sequence $s$ of distinct numbers, the {\em reduced form} of $s$, red$(s)$, is the permutation obtained from $s$ by replacing the $i$-th smallest number by $i$, $1\leq i\leq |s|$. The notion of  reduced form red$(A)$ for a block $A$ in a schematic diagram representing a permutation is defined in the same way since any block represents a sequence of distinct numbers.  

Even though we are not able to find the distributions of the patterns

\begin{center}
\begin{small}
Nr.\ 48 = $\pattern{scale = 0.6}{2}{1/1,2/2}{0/1,1/2,0/0,2/1,2/2}$\ , 
Nr.\ 49 = $\pattern{scale = 0.6}{2}{1/1,2/2}{0/1,1/2,0/0,1/1,2/0}$\ ,
Nr.\ 50 = $\pattern{scale = 0.6}{2}{1/1,2/2}{0/1,1/2,0/0,1/1,2/2}$\ ,
\end{small}
\end{center}

\noindent
in this section we will show in a bijective way that  the pattern Nr.\ 48 is equidistributed with the pattern Nr.\ 49. Additionally, note that in Table~\ref{tab-2}  we conjecture that all three patterns Nr.\ 48, Nr.\ 49 and Nr.\ 50 have the same distribution. 

%

\begin{thm}\label{eqdis-patterns-48-49}
	The patterns $p_1=\pattern{scale = 0.6}{2}{1/1,2/2}{0/1,1/2,0/0,2/1,2/2}$    and  $p_2=\pattern{scale = 0.6}{2}{1/1,2/2}{0/1,1/2,0/0,1/1,2/0}$ are equidistributed.
\end{thm}

\begin{proof} We refer to Example~\ref{ex-eqdis-patterns-48-49} illustrating our bijective proof. 

We begin with considering permutations with at least one occurrence of $p_1$ or $p_2$, and we will describe a map $g$ that sends bijectively a permutation with $k$ occurrences of $p_1$ to a permutation with $k$ occurrences of $p_2$.
 
It is straightforward to see that any element of a permutation can be involved in at most one occurrence of the pattern $p_1$. The same is true for the pattern $p_2$. Looking more closely at the occurrences of $p_1$ and $p_2$, we see that they must appear like in Fig.~\ref{dis-pic-pat-48} and~\ref{dis-pic-pat-49}, respectively, where we demonstrate the structure on three occurrences, but any other number of occurrences will clearly follow the same patten. In Fig.~\ref{dis-pic-pat-48}, the three occurrences of $p_1$, from left to right, are given by $x_iy_i$, and it is easy to see that $X_1$ and $X_3$ can be any permutations, and $X_2$ and $A$ must be $p_1$-avoiding for $X\in\{B,C,D\}$. Similarly, in Fig.~\ref{dis-pic-pat-49}, the three occurrences of $p_2$, from bottom to upward, are given by $x'_iy'_i$, and it is easy to see that $X'_1$ and $X'_3$ can be any permutations, and $X'_2$ and $A'$ must be $p_2$-avoiding for $X\in\{B,C,D\}$. 

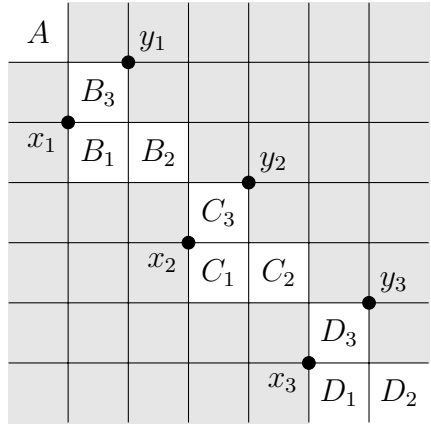
\begin{figure}[!htb]
	\begin{center}
		\begin{tikzpicture}[scale=0.8, baseline=(current bounding box.center)]
		\foreach \x/\y in {0/0,0/1,0/2,0/3,0/4,0/5,1/0,1/1,1/2,1/3,1/6,2/0,2/1,2/2,2/3,2/5,2/6,3/0,3/1,3/4,3/5,3/6,4/0,4/1,4/3,4/4,4/5,4/6,5/2,5/3,5/4,5/4,5/5,5/6,6/1,6/2,6/3,6/4,6/5,6/6}		    
		\fill[gray!20] (\x,\y) rectangle +(1,1);
		\draw (0.01,0.01) grid (6+0.99,6+0.99);
		\filldraw (1,5) circle (3pt) node[below left] {$x_1$};
		\filldraw (2,6) circle (3pt) node[above right] {$y_1$};
		\filldraw (3,3) circle (3pt) node[below left] {$x_2$};
		\filldraw (4,4) circle (3pt) node[above right] {$y_2$};
		\filldraw (5,1) circle (3pt) node[below left] {$x_3$};
		\filldraw (6,2) circle (3pt) node[above right] {$y_3$};
		\node  at (0.5,6.5) {$A$};
		\node  at (1.5,4.5) {$B_1$};\node  at (2.5,4.5) {$B_2$}; \node  at (1.5,5.5) {$B_3$};
		\node  at (3.5,2.5) {$C_1$};\node  at (4.5,2.5) {$C_2$}; \node  at (3.5,3.5) {$C_3$};
		\node  at (5.5,0.5) {$D_1$};\node  at (6.5,0.5) {$D_2$}; \node  at (5.5,1.5) {$D_3$};
		\end{tikzpicture}
	\end{center}
	\caption{Related to the proof of Theorem~\ref{eqdis-patterns-48-49}; a permutation $\pi$ with three occurrences of the pattern $p_1$}\label{dis-pic-pat-48}
\end{figure}

Given a permutation $\pi\in S_n$ with three occurrences of $p_1$ that is shown schematically in Fig.~\ref{dis-pic-pat-48}, we explain how to find the permutation $\sigma=g(\pi)\in S_n$ with three occurrences of $p_2$ that is shown schematically in  Fig.~\ref{dis-pic-pat-49}. Our description will be easily extendable to a larger, or fewer, number of occurrences of $p_1$ and $p_2$, and the fact that $g$ is bijective will be not difficult to see, so we omit its proof. Finally, the description uses the fact, to be justified at the end of the proof, that $|S_n(p_1)|=|S_n(p_2)|$, so the bijective function $f$ from $S_n(p_1)$ to $S_n(p_2)$, sending lexicographically smallest elements to lexicographically smallest elements, is well defined.

\begin{itemize}
\item Let $f$(red($A$)) be $A'$.
\item The sub-permutation of $\sigma$ formed by $\{X'_1,X'_2,X'_3\}$, for $X\in \{B,C,D\}$, is obtained from $\{X_1,X_2,X_3\}$ by the composition of rotation $180^{\degree}$ and then replacing $X_2$ by $f$(red($X_2$)) while keeping the same relative order with the elements in  $X'_1$ (see Example~\ref{ex-eqdis-patterns-48-49} below). Observe that this step is easily reversible. 
\item Note that the positions and values of $x'_i$ and $y'_i$, $i=1,2,3$, in $\sigma$ are uniquely determined once $A'$ and $X'_i$s for $X\in\{B,C,D\}$, are defined. 
\end{itemize}

\begin{figure}[!htb]
	\begin{center}
		\begin{tikzpicture}[scale=0.8, baseline=(current bounding box.center)]
		\foreach \x/\y in {0/0,0/1,0/2,0/3,0/4,0/5,1/0,1/1,1/2,1/3,1/5,1/6,2/0,2/1,2/3,2/4,2/5,2/6,3/1,3/2,3/3,3/4,3/5,3/6,4/0,4/3,4/4,4/5,4/6,5/0,5/1,5/2,5/5,5/6,6/0,6/1,6/2,6/3,6/4}		    
		\fill[gray!20] (\x,\y) rectangle +(1,1);
		\draw (0.01,0.01) grid (6+0.99,6+0.99);
		\filldraw (1,5) circle (3pt) node[above left] {$x'_3$};
		\filldraw (6,6) circle (3pt) node[above left] {$y'_3$};
		\filldraw (2,3) circle (3pt) node[above left] {$x'_2$};
		\filldraw (5,4) circle (3pt) node[above left] {$y'_2$};
		\filldraw (3,1) circle (3pt) node[above left] {$x'_1$};
		\filldraw (4,2) circle (3pt) node[above left] {$y'_1$};
		\node  at (3.5,0.5) {$A'$};
		\node  at (6.5,6.5) {$B'_1$};	\node  at (0.5,6.5) {$B'_2$};	 \node  at (6.5,5.5) {$B'_3$};
		\node  at (5.5,4.5) {$C'_1$};	\node  at (1.5,4.5) {$C'_2$}; 	\node  at (5.5,3.5) {$C'_3$};
		\node  at (4.5,2.5) {$D'_1$};	\node  at (2.5,2.5) {$D'_2$}; 	\node  at (4.5,1.5) {$D'_3$};
		\end{tikzpicture}
	\end{center}
	\caption{Related to the proof of Theorem~\ref{eqdis-patterns-48-49}; a permutation $\pi$ with three occurrences of the pattern $p_2$}\label{dis-pic-pat-49}
\end{figure}

We conclude the proof by showing that $p_1$ and $p_2$ are Wilf-equivalent. Let $a_{1,n}$ (resp., $a_{2,n}$) be the number of $n$-permutations avoiding $p_1$ (resp., $p_2$), and $b_{1,n}$ (resp., $b_{2,n}$) be the number of $n$-permutations containing at least one occurrence of $p_1$ (resp., $p_2$). Clearly, \begin{equation}\label{simple-observ}a_{i,n}+b_{i,n}=n!\mbox{ for } i=1,2.\end{equation}

We proceed by the strong version of induction on $n$. Note that $a_{1,1}=a_{2,1}=1$ and we can assume that for all $m<n$, $a_{1,m}=a_{2,m}$. We want to prove that  $b_{1,n}=b_{2,n}$, which will imply that $a_{1,n}=a_{2,n}$ by (\ref{simple-observ}).  However, this is an immediate corollary of the structures presented in Fig.~\ref{dis-pic-pat-48} and~\ref{dis-pic-pat-49} and the description of $g$, taking into account that we have assumed $|S_m(p_1)|= |S_m(p_2)|$ for $m<n$.
\end{proof}

\begin{exa}[To support the proof of Theorem~\ref{eqdis-patterns-48-49}]\label{ex-eqdis-patterns-48-49}  
By mapping lexicographically, we have $f(1)=1$; $f(21)=21$; $f(132)=231$ and $f(321)=321$. We now demonstrate the  application of $g$ to the permutation $\pi=(15)(17)(16)9(10)6(12)8(13)(11)(14)745321$ having two occurrences of $p_1$. 

We have $f(${\em red}$(A))=f(132)=231=A'$; $6,8\in B_1$, $7\in B_2$, and $(10), (11), (12), (13)\in B_3$; $(10),(12)\in B'_1$, $(11)\in B'_2$, and $5,6,7,8\in B'_3$; $C_1=C_3=C'_1=C'_3=\emptyset$; $1,2,3\in C_2$ and $(15),(16),(17)\in C'_2$; $x_1=9$,  $x_2=4$,  $y_1=14$,  $y_2=5$,  $x'_1=4$,  $x'_2=(13)$,  $y'_1=9$,  $y'_2=(14)$. Thus, the desired permutation with two occurrences of $p_2$ is\\[-4mm] $$\sigma=g(\pi)=(17)(16)(15)(13)(11)4231975(10)6(12)8(14).$$
\end{exa}

\section{Concluding remarks}

We refer to Table~\ref{tab-2} for our conjectured equidistributions.  Note that the structure of permutations with $k$ occurrences of the pattern Nr.\ 50 = $\pattern{scale = 0.6}{2}{1/1,2/2}{0/1,1/2,0/0,1/1,2/2}$ is as given in Fig.~\ref{dis-pic-pat-50} for $k=3$, where $X_1$ and $X_3$ can be any permutations, and $X_2$ and $A$ must be $\pattern{scale = 0.6}{2}{1/1,2/2}{0/1,1/2,0/0,1/1,2/2}$ - avoiding for $X\in\{B,C,D\}$. Even though the structure in Fig.~\ref{dis-pic-pat-50} is very similar to those in Fig.~\ref{dis-pic-pat-48} and~\ref{dis-pic-pat-49} corresponding to the patterns Nr.\ 48 and Nr.\ 49, respectively, we were not able to find a bijective proof showing the conjectured equidistribution of the three patterns. Indeed, there is a problem with $\pattern{scale = 0.6}{2}{1/1,2/2}{0/1,1/2,0/0,1/1,2/2}$ - avoiding blocks $X_2$ in Fig.~\ref{dis-pic-pat-50}, for $X\in\{B,C,D\}$, having relations with both $X_1$ and $X_3$, horizontally and vertically, while in Fig.~\ref{dis-pic-pat-48} and~\ref{dis-pic-pat-49} the respective blocks $X_2$ have only horizontal relations with $X_1$, and $X_1$ and $X_3$ have vertical relations. 

\begin{figure}[!htb]
	\begin{center}
		\begin{tikzpicture}[scale=0.8, baseline=(current bounding box.center)]
		\foreach \x/\y in {0/0,0/1,0/2,0/3,0/4,0/5,1/0,1/1,1/2,1/3,1/5,1/6,2/0,2/1,2/2,2/3,2/6,3/0,3/1,3/3,3/4,3/5,3/6,4/0,4/1,4/4,4/5,4/6,5/1,5/2,5/3,5/4,5/4,5/5,5/6,6/2,6/3,6/4,6/5,6/6}		    
		\fill[gray!20] (\x,\y) rectangle +(1,1);
		\draw (0.01,0.01) grid (6+0.99,6+0.99);
		\filldraw (1,5) circle (3pt) node[above left] {$x_1$};
		\filldraw (2,6) circle (3pt) node[above left] {$y_1$};
		\filldraw (3,3) circle (3pt) node[above left] {$x_2$};
		\filldraw (4,4) circle (3pt) node[above left] {$y_2$};
		\filldraw (5,1) circle (3pt) node[above left] {$x_3$};
		\filldraw (6,2) circle (3pt) node[above left] {$y_3$};
		\node  at (0.5,6.5) {$A$};
		\node  at (1.5,4.5) {$B_1$};\node  at (2.5,4.5) {$B_2$}; \node  at (2.5,5.5) {$B_3$};
		\node  at (3.5,2.5) {$C_1$};\node  at (4.5,2.5) {$C_2$}; \node  at (4.5,3.5) {$C_3$};
		\node  at (5.5,0.5) {$D_1$};\node  at (6.5,0.5) {$D_2$}; \node  at (6.5,1.5) {$D_3$};
		\end{tikzpicture}
	\end{center}
	\caption{Related to Pattern Nr.\ 50}\label{dis-pic-pat-50}
\end{figure}
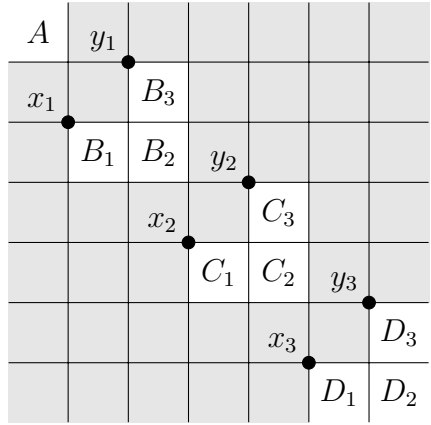

Additionally, we state the following conjecture.

\begin{conj}\label{conj-nr-3} 
{\em The distribution of the pattern Nr.\ $3$ = $\pattern{scale = 0.6}{2}{1/1,2/2}{0/0,0/1,1/2}$ is given by \cite[$A200545$]{OEIS}, which is 	the triangle, read by rows, given by \\[-3mm]
$$(1,0,2,1,3,2,4,3,5,4,6,5,7,6,\ldots)\ \mbox{\em DELTA } (0,1,0,1,0,1,0,1,0,1,\ldots)$$ where {\em DELTA} is the operator defined in $A084938$ in \cite{OEIS} in terms of continued fraction: the triangle $[r_0,r_1,\ldots]\  \mbox{{\em DELTA}}\  [s_0,s_1,\ldots]$ has generating function $$\frac{1}{1-\frac{r_0x+s_0xy}{1-\frac{r_1x+s_1xy}{1-\frac{r_2x+s_2xy}{1-\ldots}}}}.$$}
\end{conj}

Note that the operator DELTA was already linked to patterns in permutations, and also to so-called {\em Riordan arrays}, in \cite[$A200545$]{OEIS}.

As a direction for further research, we suggest studying joint distribution of patterns considered in this paper and other permutation statistics. As an illustration of this idea, we derive the following generating function
$$F(x,q,t)=\sum_{n\geq 0}x^n\sum_{\pi\in S_n}q^{\pattern{scale=0.5}{1}{1/1}{0/1,1/0}(\pi)}t^{\mbox{\tiny des}(\pi)}$$
generalizing Theorem~\ref{thm-length-1}. Let $$F(x,t)=\sum_{n\geq 0}x^n\sum_{\pi\in S_n}t^{\mbox{\tiny des}(\pi)}=\sum_{n\geq 0}A_n(t)x^n,$$
where $A_n(t)$'s are the well known {\em Eulerian polynomials}. Also, let $$F(x,t)=\sum_{n\geq 0}x^n\sum_{\pi\in S_n(\hspace{-1.5mm}\pattern{scale=0.5}{1}{1/1}{0/1,1/0})}t^{\mbox{\tiny des}(\pi)},$$ where recall that  $S_n(\hspace{-1.5mm}\pattern{scale=0.5}{1}{1/1}{0/1,1/0})$ denotes the set of \pattern{scale=0.8}{1}{1/1}{0/1,1/0}-avoiding permutations. Now, following the same steps as in the proof of Theorem~\ref{thm-length-1}, we obtain: 
\begin{align*}
G(x,t)+xG(x,t)F(x,t)=F(x,t)\ \ \ \ \ & \Rightarrow \ \ \ \ \ G(x,t)=\frac{F(x,t)}{1+xF(x,t)};\\
G(x,t)+xqG(x,t)F(x,q,t)=F(x,q,t)\ & \Rightarrow \ \ F(x,q,t) = \frac{F(x,t)}{1+x(1-q)F(x,t)}.
\end{align*}

As a final remark, we note that it would be interesting to classify completely mesh patterns of length 2 with respect to their distribution. As noted in the introduction, the number of equivalence classes here is  larger than that of Wilf-equivalence classes (given by equivalence with respect to avoidance) discussed in \cite{Hilmarsson2015Wilf}.

\section*{\bf Acknowledgments}
The authors are grateful to the anonymous referee for providing several suggestions that helped to improve the presentation of the results. The second author was partially supported by the National Science Foundation of China (Nos. 11626172, 11701424).

\end{document}